\documentclass[reqno]{amsart}

\usepackage{graphicx,pdfsync,amssymb,mathrsfs,amsfonts,amsbsy,color,esint,mathtools,mdframed}
\usepackage{tikz} \usetikzlibrary{decorations.pathreplacing}

\usetikzlibrary{positioning}
\usepackage[shortlabels]{enumitem}
\usepackage{makecell}

\usepackage{booktabs}
\usepackage{array}
\usepackage[table]{xcolor}
\usepackage{hyperref}

\hypersetup{
colorlinks=true,
linkcolor=blue,
filecolor=blue,
urlcolor=blue,
citecolor=blue,
}

\usepackage[utf8]{inputenc}
\usepackage[T1]{fontenc}

\DeclareMathOperator{\supp }{supp}

\DeclareMathOperator{\dist}{dist}

\newtheorem{theorem}{Theorem}[section]
\newtheorem{lemma}[theorem]{Lemma}
\newtheorem{proposition}[theorem]{Proposition}
\newtheorem{definition}[theorem]{Definition}
\newtheorem{corollary}[theorem]{Corollary}
\newtheorem{remark}[theorem]{Remark}

\newtheorem{question}[theorem]{Question}
\newtheorem{problem}[theorem]{Problem}

\setlist[enumerate]{itemsep=3pt}
\setlist[itemize]{itemsep=3pt}

\def \T  {\mathbb{T}} 
\def \R {\mathbb{R}}  
\def \N {\mathbb{N}}  

\def \p {\partial}

\def \ep {\epsilon}
\def \om {\omega}
\def \Om {\Omega}

\numberwithin{equation}{section}

\begin{document}

\title[Flexibility and rigidity for Couette]{Flexibility and rigidity for the Couette flow in the infinite channel}

\author{Dengjun Guo}

\address{Academy of Mathematics and Systems Science, Chinese Academy of Sciences, Beijing, China}

\email{djguo@amss.ac.cn}

\author{Xiaoyutao Luo}

\address{State Key Laboratory of Mathematical Sciences, Academy of Mathematics and Systems Science, Chinese Academy of Sciences, Beijing, China}

\email{xiaoyutao.luo@amss.ac.cn}

\author{Guolin Qin}

\address{State Key Laboratory of Mathematical Sciences, Academy of Mathematics and Systems
Science, Chinese Academy of Sciences, Beijing, China}

\email{qinguolin18@mails.ucas.ac.cn}

\date{\today}

\begin{abstract}
We investigate the existence of stationary and traveling wave solutions to the 2D Euler equations near the Couette flow in the infinite channel $\mathbb{R} \times [-1,1]$. For Sobolev spaces $W^{s,p}$ or H\"older spaces $C^s$, we identify the index $s=  1+  \frac1p $   as the vorticity regularity threshold separating  flexibility from rigidity. Specifically, for any $s<1+  \frac1p$ we prove the existence of $C^\infty$ smooth, compactly supported steady states and traveling waves arbitrarily close to the Couette flow in all $W^{s,p}$ and $C^{1-}$. Conversely, we establish the non-existence of such relative equilibria in $ W^{s,p}$ with $s>1+  \frac1p$ or $C^{1+}$. A notable feature of the variational construction is that these flexible solutions belong to every Gevrey class strictly below the analytic threshold.
\end{abstract}

\maketitle

\section{Introduction}\label{sec:intro}

Shear flows are important coherent structures of the 2D Euler equations. In this paper, we investigate  one of the simplest shear flows, the inviscid Couette flow $(y,0)$ on an infinite channel $\R \times [-1,1]$.

Consider the vorticity perturbation $\om  : \Omega   \to \R$ near Couette flow $(y,0)$ satisfying the 2D Euler equations: 
\begin{equation}\label{eq:eu_eq couette}
\begin{cases}
 \p_t \om +  y\p_x \om + u \cdot \nabla \om = 0 & \\
	\om|_{t= 0 } = \om_{in}   &
\end{cases}	
\end{equation}
In this paper, we restrict our attention to the steady setting:
\begin{equation}\label{eq:steady_euler}
  y\p_x \om + u \cdot \nabla \om = 0 .
\end{equation}

The study of the 2D Euler equations near a specific background state is characterized by a deep dichotomy between flexibility and rigidity. In the literature of fluid dynamics, rigidity refers to the phenomenon where the nonlinear structure of the equations, combined with boundary or geometric conditions, forces any steady solution close to a background state to either be trivial or strictly inherit the symmetries of that background \cite{HamelNadir17, HamelNadir19, GSPSY_21, Ruiz_23, FWZ_24,  WZ_24}. 

Conversely, flexibility describes the failure of this rigidity, allowing for the existence of nontrivial relative equilibria that remain arbitrarily close to the background flow \cite{LinZeng2011,CZEW_23,FMM24}.

The inviscid Couette flow $(y, 0)$ is one of the most studied shear flows of the 2D Euler equations since   Kelvin~\cite{Kelvin1887}. The primary objective of this paper is to investigate the dichotomy between flexibility and rigidity in the steady 2D Euler equations near this specific background state.

This leads to the following basic question:
\begin{question}
Given a Banach space $X$, if $\| \om \|_{X} \ll 1$, is it necessary that $\om=0$?
\end{question}

Since the transport dynamics of the 2D Euler equations are highly sensitive to regularity, the answer depends strongly on the choice of the space $X$. This leads to a more refined formulation in terms of a regularity parameter. Given a family of Banach spaces $X_\alpha$ characterized by a regularity index $\alpha$ (e.g. $H^{\alpha}$, $C^{\alpha}$ etc.),  we seek to identify a critical regularity threshold $\alpha_c$.
\begin{question}
Does there exist a threshold $\alpha_c$ such that $\alpha> \alpha_c $ and $\| \om \|_{X_\alpha } \ll 1$ implies $\om =0$?
\end{question}

 \subsection{Motivation}
Our first motivation comes from the time-dependent problem \eqref{eq:eu_eq couette} and the question of asymptotic convergence. In \cite{GL_2026}, the first two authors proved asymptotic stability for a class of unidirectional shear flows  for which perturbations converge to zero as time tends to infinity, provided there is no interior stagnation point in the background shear.

In our setting on $\mathbb{R} \times [-1,1]$, the Couette flow $(y, 0)$ changes sign across the domain, creating a stagnation line at $y=0$ where fluid particles are at rest. Determining whether nontrivial steady states can exist on this stagnation line clarifies whether the assumption of ``no interior stagnation'' in \cite{GL_2026} is mathematically sharp.

A second motivation comes from the broader program of classifying steady solutions to the 2D Euler equations~\cite{CV_2012,HamelNadir17,HamelNadir19,HamelNadir23,DeregibusEspositoRuiz2025,EHSX_2026}. While these prior works primarily derive rigidity from various geometric conditions to understand the global structure, we focus on the local structure near a particular stationary solution and emphasize the role of regularity.

This regularity-based perspective has already played an important role in the study of inviscid damping. In the pioneering work \cite{LinZeng2011} by Lin and Zeng, it was demonstrated that for the Couette flow on the periodic channel,  non-shear steady solutions exist in $H^{3/2-}$, whereas no non-shear steady or traveling wave solutions exist in $H^{3/2+}$. Notably, \cite{LinZeng2011} can be seen as an improvement of the earlier work~\cite{BE1999} of Burton and Emamizadeh where compactly supported steady vortices near the Couette flow in $\R^2$  are constructed. However, solutions in~\cite{BE1999} are only in $L^p$, which is  insufficient to identify a stability threshold.

Our work extends this regularity-based classification for the Couette flow in the infinite channel $\mathbb{R} \times [-1,1]$. More importantly, we provide a complete classification across the whole range of Sobolev class $W^{1+1/p, p } $.

\subsection{Existence of relative equilibria} 

We postpone a detailed discussion of related work and first state the existence result for nearby steady solutions. 

\begin{theorem}[Flexibility]\label{thm:flex}
Fix $0 < \delta \leq 1$. There exists $\ep_\delta>0 $ with the following property. 

There exists a family of functions $\om_{\ep}: \Om  \to \R  $ for   $0< \ep\leq \ep_\delta$  such that  $\om_{\ep} \in C^\infty_c(\Om)$ is a steady solution to \eqref{eq:eu_eq couette} with \begin{equation}\label{eq:thm:flex}
    \lim_{\ep \to 0+ } \| \om_\ep \|_{W^{
     1+ \frac{1}{p}-\delta,p}   }  =0  \qquad \text{for any $ 1\leq p\leq \infty$}.
\end{equation}

Moreover, for any $0<\ep \le \ep_{\delta}$, the following asymptotic estimates hold:
\begin{itemize}
    \item Anisotropic support: 
    \begin{equation}
    |\supp_y (\om_\ep)| \lesssim \ep |\log \ep|^{\frac{1}{2}}, \qquad |\supp_x  (\om_\ep)|\lesssim |\log \ep|^{-1}.
\end{equation}

\item Quantitative regularity:
\begin{equation}
    | \nabla^k \om_\ep |_{L^{\infty}} \lesssim_{k,\delta} \ep^{1-\delta-k} \quad \text{for any $k \in \N $}.
\end{equation}

\item Qualitative smoothness: the solutions $\om_\ep$ belong to the Gevrey classes ${G}^{1+\delta} $.
\end{itemize}

\end{theorem}

The variational construction used to prove Theorem \ref{thm:flex} is not restricted to stationary solutions. After shifting the background shear potential, the same argument produces traveling waves with any speed \(c\in(-1,1)\).

\begin{corollary}\label{cor:flex}
Fix $0 < \delta\leq 1$ and $c \in (-1, 1)$. There exist $\ep_0=\ep_0(\delta,c)>0$ and  a family of functions $\om_{\ep}: \Om \to \R  $ for   $0< \ep \leq \ep_0$ such that $\om_{\ep} \in C^\infty_c(\Om)$ and $\om_\ep (x - ct, y)$ is a traveling wave solution to \eqref{eq:eu_eq couette} with
\begin{equation}\label{eq:cor:flex}
 \lim_{\ep \to 0+ } \| \om_\ep \|_{W^{
     1+ \frac{1}{p}-\delta ,p}  }  =0  \qquad \text{for any $ 1\leq p\leq \infty$}.
\end{equation}

Moreover, the same anisotropic support bounds, derivative estimates, and Gevrey regularity stated in Theorem \ref{thm:flex} hold for this family as \(\ep\to0\).

\end{corollary}

\begin{remark} 
We highlight several features of the construction that distinguish it from previous flexibility results.
    \begin{itemize}

    \item Previous flexibility constructions for nearby shear flows were carried out on bounded or periodic channels  \(\mathbb T\times[-1,1]\) ~\cite{LinZeng2011,CZEW_23,DrivasNualart2024} using different methods.  In contrast, our construction   on the infinite channel $\mathbb R\times[-1,1]$ lacks compactness in the horizontal direction. This requires  localization in $x$, concentration near the stagnation line $y=0$, and smallness in the relevant Sobolev scales.

    \item Although variational methods have been widely used to construct steady solutions of the 2D Euler equations, thus far many constructions only possess a finite regularity. In contrast, our penalized variational problem produces a smooth nonlinear vorticity--stream function relation.  To our knowledge, this gives the first Euler variational construction with Gevrey regularity, a regularity question raised in \cite{AbeChoiJeongSimWoo2026}.

\item Compared with the construction of Burton and Emamizadeh \cite{BE1999}, several essential distinctions separate the two works. First, their solutions are constructed in $\mathbb R^2$ by maximization over rearrangement classes, whereas our construction operates in the infinite channel with a penalized energy. Second, the rearrangement framework in \cite{BE1999} yields only $L^p$ regularity, $p>2$,  while our penalized method, by contrast, produces smooth and Gevrey-regular solutions which are small in  $W^{1+1/p-,p}$.

        \item We emphasize  that the geometry of the infinite channel plays a decisive role. The Couette flow has a stagnation line $y=0$, and our steady states are localized precisely near this critical layer. The construction produces a strongly anisotropic profile: the vertical support is much thinner than the horizontal one, while the amplitude remains small enough to approach the threshold $W^{1+1/p,p}$.  

    \end{itemize}
\end{remark}

\subsection{Triviality of relative equilibria}

In the opposite direction, higher regularity forces the perturbation to vanish. The key issue is regularity across the critical layer $y=c$, where the relative horizontal transport speed vanishes.  

\begin{theorem}[Rigidity: Sobolev scale]\label{thm:rigid_intro_1}
For any $1< p < \infty$ and $s> 1+\frac{1}{p}$  there exists $\ep_s>0$ such that the following holds. If $\om \in L^2(\Omega)  $ is a traveling wave solution to \eqref{eq:eu_eq couette} with speed $c \in \R$ and
\begin{equation}\label{eq:thm:rigid}
    \| \om \|_{W^{s,p} (\Omega)} \leq \ep_s,
\end{equation}
then $\om =0$.
\end{theorem}

We also establish rigidity in the classical H\"older scale, confirming that $C^{1}$ is the critical threshold.

\begin{theorem}[Rigidity: H\"older scale]\label{thm:rigid_intro_2}
For any $\alpha> 0$ there exists $\ep_\alpha >0$ such that the following holds. If $\om \in L^2(\Omega) $ is a traveling wave solution to \eqref{eq:eu_eq couette} with speed $c \in \R$ and
\begin{equation}\label{eq:thm:rigid2}
    \| \om \|_{C^{1,\alpha}(\Omega)} \leq \ep_\alpha,
\end{equation}
then $\om =0$.
\end{theorem}

\begin{remark}
    Theorems \ref{thm:flex}, \ref{thm:rigid_intro_1}, and \ref{thm:rigid_intro_2} together identify the regularity $s=1+\frac1p$ in $W^{s,p}$ as the flexibility threshold for the Couette flow. This demonstrates the distinct 1D nature of the problem and stands in sharp contrast to the Poiseuille flow or Kolmogorov flow~\cite{CZEW_23,DrivasNualart2024}.
\end{remark}
\subsection{Related results}

The study of 2D Euler equations near shear flows has a long history, particularly in connection with asymptotic stability, inviscid damping, and steady solutions. 

\subsubsection*{Flexibility, rigidity, and inviscid damping}

Our work connects the local variational construction of steady states with the flexibility--rigidity dichotomy near shear flows. The behavior of perturbations near fundamental shear flows is intimately tied to regularity, a fact central to the theory of inviscid damping \cite{BedrossianMasmoudi15}. For the Couette flow on a periodic channel, Lin and Zeng \cite{LinZeng2011} demonstrated a sharp regularity threshold $H^{3/2}$ that separates the existence of smooth, localized relative equilibria from rigidity. Since these rigid states do not decay in time, they provide the first counter-examples to inviscid damping in lower order Sobolev spaces. Our results extend this analysis \cite{LinZeng2011} to the infinite channel and to the full $L^p$-Sobolev class, identifying the threshold $W^{1+1/p,p}$ for the flexibility--rigidity dichotomy.

We emphasize that the flexibility/rigidity threshold depends not only on the underlying physical domain, but also on the interplay between the background steady state and the spatial geometry. Using   linearization and bifurcation methods around suitable shear profiles, it was shown in \cite{CZEW_23} that the Poiseuille flow is rigid in $H^{5+}(\T\times [-1,1])$, while the Kolmogorov flow on $\T^2$ is flexible even in the analytic class. Later, in \cite{DrivasNualart2024}, an exact threshold $C^2(\T\times [-1,1])$ was identified for the Poiseuille flow on the periodic channel, while for all integers $n\ge 2$, flexibility below $C^{n}$ was demonstrated for the power-law shear $y^n$. In \cite{DrivasNualart2024}, the velocity formulation is used together with the local stagnation structure, by inserting compactly supported radial vortices into the stagnation region.  These findings show that the flexibility/rigidity threshold is sensitive to both the geometry of the domain and the degeneracy structure of the underlying shear. In the time-dependent setting, related linear mechanisms also appear in the study of metastability and inviscid damping   \cite{LX_19}.

Furthermore, structural rigidity---whereby steady solutions are forced to inherit domain or background symmetries---has been actively classified in various geometries \cite{GSPSY_21,HamelNadir17,CDG_21}.

\subsubsection*{Variational methods for steady flows}
The mathematical construction of steady 2D Euler flows frequently employs variational principles, a tradition tracing back to Kelvin and Arnold \cite{Arnold1966}. The ``vorticity method'' of constructing steady flows  was rigorously developed by Benjamin \cite{Benjamin1976}, Turkington \cite{Turkington1983}, and Burton \cite{Burton1987, Burton1989}. See \cite{FraenkelBerger1974,Norbury1975,FriedmanTurkington1981} for parallel variational developments for vortex rings and multiple vortex configurations.

Recent developments have refined these variational methods for concentrated vortex structures, including Hill's spherical vortex~\cite{Choi24}, vortex dipoles~\cite{AbeChoi22,ChoiJeongSim25,AbeChoiJeong2025,AbeChoiJeongSimWoo2026}, and other vortex configurations~\cite{CWW_20,MR4295232,CWWZ_22,CW2024}.

While the aforementioned literature typically focuses on the global structure  or orbital stability of isolated vortex configurations, our work adapts the variational framework in local neighborhood of the Couette flow, similar to the earlier work  \cite{BE1999} of Burton and Emamizadeh. Unlike the rearrangement setup in \cite{BE1999}, the specific penalized energy functional we engineer allows us to achieve the anisotropic scaling necessary to construct $C^\infty$ smooth, compactly supported relative equilibria that are arbitrarily small in $W^{1+ 1/p,p }$.  The Gevrey smoothness of the vorticity function then naturally leads to the Gevrey smoothness of the final solution.

\subsubsection*{Classification of steady Euler flows}

Complementing our flexibility construction, rigidity for steady Euler flows has also been the subject of recent work. Various geometric and analytical conditions have been established that force solutions to inherit the symmetry of the domain or the background flow, effectively classifying global steady states \cite{GSPSY_21,HamelNadir17,HamelNadir19,HamelNadir23,GXX_24, Ruiz_23, FWZ_24, WZ_24}.

Our rigidity result operates locally, precisely identifying the functional thresholds ($W^{1+1/p+,p}$ and $C^{1+}$) where the nonlinear transport structure rules out small, localized steady perturbations near the critical layer and enforces triviality.

\subsection{Main ideas of the proof}
The proofs have two main components: a variational construction of relative equilibria in low regularity, and transport-driven elliptic estimates that prove triviality at higher regularities.

\subsubsection*{Flexibility via variational constructions}
To prove the existence of nearby steady solutions, we adopt a variational method that has drawn much attention recently~\cite{AbeChoi22,Choi24,ChoiJeongSim25,AbeChoiJeongSimWoo2026,MR4295232}.

The method is based on a maximization problem under suitable constraints. In our setting, the energy consists of three parts~\footnote{For ease of exposition, the signs here are different from the ones used in the construction.}:
\begin{equation}\label{eq:E_max_intro}
\begin{aligned}
    \mathcal{E}(\om) & =\frac{1}{2} \int \om \mathcal{G}\om + \frac{1}{2} \int y^2 \om  + \int F_\ep(\om)  \\
    &:= \underbrace{\mathcal{E}_1(\om)}_{\text{self energy}} +  \underbrace{\mathcal{E}_2(\om)}_{\text{interaction energy}} + \underbrace{\mathcal{E}_3(\om) .}_{\text{penalty energy}}
\end{aligned}
\end{equation}
Using a suitably chosen admissible class, the maximizer yields a steady solution to \eqref{eq:eu_eq couette}. Constructing a solution with smallness in the maximal possible regularity class ($W^{1+1/p -,p}$ and $C^{1-}$) imposes severe structural constraints:

\begin{itemize}
\item \textbf{Amplitude and Sign:} The perturbation must have small $L^\infty$ magnitude in order to remain small in higher Sobolev spaces. Furthermore, the vorticity must be strictly negative ($\omega \le 0$); otherwise, the linear interaction energy $\mathcal{E}_2$ would strictly dominate, driving the vorticity toward the boundaries $y=\pm 1$ and forcing it to vanish.

    \item \textbf{Anisotropy:} The vertical scaling of the solution must be much smaller than its horizontal scaling. Without this anisotropy, the interaction energy $\mathcal{E}_2$ (which scales linearly with amplitude) would overwhelm the self-energy $\mathcal{E}_1$ (which scales quadratically), resulting in a strictly negative total energy and precluding the existence of a localized maximizer.

    \item \textbf{Smoothness:} The penalty term $\mathcal{E}_3$ is designed to keep the solution in the desired small $L^\infty$ regime and to enforce a smooth, $C^\infty$ transition at the boundary of the vortex support.

\end{itemize}

By precisely tuning this anisotropic scaling, the resulting maximizer approaches the $W^{1+1/p,p}$ threshold. Compared with the construction in the periodic channel~\cite{LinZeng2011}, our solution exhibits a clear anisotropic pattern and is more explicit.

Traveling waves with speed $c\in(-1,1)$ are obtained by shifting the background shear potential, equivalently replacing the $y^2$ term in $\mathcal{E}_2$ by the corresponding shifted potential. The exclusion $c=\pm1$ is not merely technical: otherwise one would obtain steady solutions near $(y,0)$ on the channel $\R\times[0,1]$, contradicting the asymptotic convergence result in~\cite{GL_2026}.

\subsubsection*{Rigidity via critical layer analysis}

The rigidity argument in $W^{1+1/p+,p}$ relies on the singular behavior near the critical layer. For a traveling wave, the steady Euler equation dictates that $(y + c + u^x)\partial_x \omega + u^y \partial_y \omega = 0$. The ``critical layer'' occurs where the horizontal transport speed vanishes, $y + c + u^x = 0$. Our proof of non-existence of relative equilibria in $W^{1+1/p+,p}$   was inspired by the argument in~\cite{LinZeng2011}. However, the argument in~\cite{LinZeng2011} does not directly extend to the infinite channel and does not cover the full Sobolev or H\"older scales considered here.

We analyze the elliptic identity $\Delta u^y = \partial_x \omega$. Using the transport equation, the horizontal derivative of the vorticity can be rewritten as $\partial_x \omega = -u^y \partial_y \omega / (y + c + u^x)$. If the perturbation is sufficiently small in a high-regularity space, the critical layer forms a 1D curve. Smallness in $W^{1+1/p+,p}$ enforces a decay rate on the vorticity gradients that is ultimately incompatible with the width and non-integrability of the critical layer. Using Hardy--Littlewood maximal functions and fractional Sobolev embeddings, we show that the resulting energy bounds force the velocity field to vanish identically, establishing rigidity above the threshold.

 The one-dimensional nature of the critical layer is reflected both in the anisotropic scale of the flexibility construction and in the critical derivative count $1+\frac1p$. Since the $C^2$ threshold obtained for the Poiseuille flow in \cite{DrivasNualart2024} is insensitive to dimensionality, whether this anisotropic phenomenon is universal in the steady shear flow flexibility/rigidity problem remains unclear.

\subsection{Organization of the paper}

The paper is organized as follows.
\begin{itemize}
    \item In Section \ref{sec:prelim}, we introduce the functional setup, the Green function estimates, and the auxiliary Sobolev lemmas used throughout the paper.

    \item Section \ref{sec:outline} reduces the proof of the flexibility theorem to the main variational proposition, Proposition \ref{prop:maximizer}.
Section \ref{sec:flex} proves this variational proposition by constructing the maximizer and deriving its Euler--Lagrange equation and support bounds.  In Section \ref{sec:sketch traveling}, we explain the modifications needed to obtain traveling waves.

\item Finally, Section \ref{sec:rigid} proves the rigidity results, Theorems \ref{thm:rigid_intro_1} and \ref{thm:rigid_intro_2}.
\end{itemize}

\subsection*{Acknowledgment}
XL is supported by NSFC No. 12421001 and  No. 12288201. G. Qin was supported by National Key R\&D Program of China (Grant 2025YFA1018400) and NNSF of China (Grant 12471190).

\section{Preliminaries}\label{sec:prelim}

\subsection{Notations} Throughout this paper, we consider the 2D infinite channel $\Omega = \mathbb{R} \times [-1, 1]$. Points in the domain are typically denoted by $z = (x, y)$, where $x \in \mathbb{R}$ is the horizontal variable and $y \in [-1, 1]$ is the vertical variable.

We utilize standard notations for Lebesgue and Sobolev spaces. For $1 \le p \le \infty$, $L^p(\Omega)$ denotes the usual space of Lebesgue measurable functions with finite $p$-norm, denoted as $\|\cdot\|_{L^p}$.

For any measurable set $A \subset \R^n$, $|A|$ denotes its   Lebesgue measure, and $1_A$ denotes its characteristic function. The (essential) support of a measurable function $f$ is denoted by $\supp(f)$. Similarly, $\supp_y(\cdot)$ and $\supp_x(\cdot)$ are defined as the projection onto $x$- or $y$-axis for any two-variable functions.

Finally, $C$ will be used to denote a generic positive constant that may change from line to line. If a constant depends strictly on specific parameters (e.g., $\delta$ or $s$), we will indicate this dependence via subscripts, such as $C_\delta$ or $C_s$.

\subsection{Function spaces}

The space of infinitely differentiable functions with compact support in $\Omega$ is denoted by $C_c^\infty(\Omega)$.

For fractional Sobolev spaces $W^{k+s,p }(\Omega)$ with $1\le p \le \infty $, $k\in \N$, and $s \in (0,1)$, the norm is given by $\|f\|_{W^{k+s,p}(\Omega)} = \|f\|_{W^{k,p}(\Omega)} + [\nabla^k f]_{W^{s,p} (\Omega)} $, where the Gagliardo semi-norm is defined as
$$
[f]_{W^{s,p}(\Omega)} = \left( \int_\Omega \int_\Omega \frac{|f(x) - f(y)|^p}{|x - y|^{n+sp}} \, dx \, dy \right)^{1/p}.
$$
For any  $s \ge 0$ and $p=2$, $H^s(\Omega)$ denotes the   Sobolev space $W^{s,2}(\Omega)$.

We denote the space of H\"older continuous functions by $C^{k,\alpha}(\Omega)$ for $k \in \mathbb{N}$ and $\alpha \in (0,1)$, equipped with the standard norm $\|\cdot\|_{C^{k,\alpha}}$,
$$
\| f \|_{C^{k,\alpha}(\Omega)} = \| f \|_{C^{k }(\Omega)} + \sup_{z\ne z'} \frac{|\nabla^{k} f(z) -\nabla^{k} f(z') |}{|z-z'|^\alpha} .
$$

Note that   for any $k\in \N$, $ s\in(0,1)  $, the space \(W^{k+s,\infty}(\Omega) = C^{k,s}(\Omega)\).

We use the following definition of Gevrey spaces to measure the regularity of the solutions.
\begin{definition}[Gevrey spaces]
Let $s\ge 1$. We say that a function $f\in C^\infty(\Omega)$ belongs to the
Gevrey class $G^s(\Omega)$ if there exist constants $K_0,K_1>0$ such that
\[
\|\partial^\alpha f\|_{L^\infty(\Omega)}
\le K_0 K_1^{|\alpha|} (|\alpha|!)^s
\qquad \text{for all multi-indices } \alpha\in \mathbb N^2.
\]
Here $\partial^\alpha=\partial_x^{\alpha_1}\partial_y^{\alpha_2}$ and
$|\alpha|=\alpha_1+\alpha_2$.

More generally, for an interval $I\subset \mathbb R$, we say that $g\in C^\infty(I)$
belongs to $G^s(I)$ if there exist constants $K_0,K_1>0$ such that
\[
|g^{(m)}(t)| \le K_0 K_1^m (m!)^s
\qquad \text{for all } m\in \mathbb N,\ t\in I.
\]
When $s=1$, this is the real-analytic class.
\end{definition}

 For a proof of the following proposition, see for instance \cite{Friedman58,Tri_03}.

\begin{proposition}[Local Gevrey regularity without loss]\label{prop:semilinear-elliptic-gevrey}
Let $\Omega\subset\mathbb R^n$ be open, let $s>1$, and let $h\in G^s(\mathbb R)$.
Assume that $u\in C^\infty(\Omega)$ solves
\[
-\Delta u = h(u)\qquad\text{in }\Omega.
\]
Then
\[
u\in G^s_{\mathrm{loc}}(\Omega).
\]
\end{proposition}

\subsection{Stream-vorticity formulation and the Green function }
For the 2D Euler equations on the infinite channel $\Omega = \mathbb{R} \times [-1, 1]$, we  introduce a scalar stream function $\psi$, such that $u = \nabla^\perp \psi = (\partial_y \psi, -\partial_x \psi)$. The vorticity $\omega$ is related to the stream function through   
\begin{equation}
\begin{cases}
    -\Delta \psi = \omega \quad \text{in } \operatorname{int} \Omega  &\\
    \psi =0 \quad \text{on } \p \Omega &
\end{cases}
\end{equation} 
 with the usual Dirichlet boundary condition $\psi = 0$ on $\partial\Omega = \{ y = \pm 1 \}$.

The stream function can be recovered from the vorticity via the Green operator $\mathcal{G}$, defined as:$$\psi(z) = \mathcal{G}\omega(z) := \int_{\Omega} G(z, z') \omega(z') dz',$$
where $G(z, z')$ is the Green function for the negative Dirichlet Laplacian on the infinite channel $\Omega$. It is a standard fact that $G(z, z')$ is  positive for $z, z' \in \operatorname{int}(\Omega)$, and behaves like $-\frac{1}{2\pi} \log|z-z'|$ as $|z -z'| \to 0 $. Indeed,
  the Green function is given explicitly by
\begin{equation}
 G(x,y, x',y') =  \frac{1}{4\pi} \ln \left( \frac{\cosh\left[\frac{\pi}{2} (x-x')\right] - \cos\left[\frac{\pi}{2}(y+y')\right]}{\cosh\left[\frac{\pi}{2}(x-x')\right] - \cos\left[\frac{\pi}{2}(y-y')\right]} \right) .
\end{equation}
Though we will not use this precise formula in this paper.

Because $\Om$ is bounded in the $y$-direction and infinite in the $x$-direction, the kernel exhibits an isotropic singular behavior at short distances but decays exponentially at large horizontal distances. We record this fundamental property in the following lemma:
\begin{lemma}\label{lemma:kernel G}
    The kernel $G  (x,y,x',y') : \Omega  \times \Om  \to \R$ satisfies
    \begin{equation}\label{eq bound K}
        |G(z,z')| \lesssim 
    \begin{cases}
    1 - \log{|z-z'|} &\quad\text{if $|z-z'|\le 1 $  } \\
 e^{-\frac{\pi}{2}|x-x'|} &\quad\text{if $|z-z'| \ge 1 $  }. 
\end{cases}
    \end{equation}
\end{lemma}

\begin{lemma}\label{lemma:G bounds}
    Assume that $g:\Omega\to\mathbb{R}$ satisfies
    \[
    \|g\|_{L^1(\Omega)}\lesssim \ep^2,
    \qquad
    \|g\|_{L^\infty(\Omega)}\lesssim \ep^{1-q}
    \]
    for some $0<q \le \frac12$. Then there holds
    \begin{equation}\label{eq control psi}
        \int_{\Omega} |G(z,z')g(z')|\,dz' \lesssim |\log \ep| \ep^2.
    \end{equation}
    
\end{lemma}
\begin{proof} 

We use the following bound on the Green function
\begin{equation}\label{eq:aux lemma:G bounds 1}
|G(z,z')| \leq C  (1 + |\log|z-z'||) \quad \text{for all } z, z' \in \Omega .
\end{equation}
Let $B_R (z) $ be the ball of radius $R=\ep^{\frac{1+q}{2}}$ centered at $z\in \Om$. Then
\begin{equation}\label{eq:aux lemma:G bounds 2}
    \int_{\Omega} |G(z,z')g(z')|\,dz' = \int_{\Omega \cap B_R(z)} |G(z,z')g(z')|\,dz' + \int_{\Omega \setminus B_R(z)} |G(z,z')g(z')|\,dz'.
\end{equation}
In the ball $B_R(z)$, we bound $g$ by its $L^\infty$ norm; the integral can then be estimated directly:
\begin{align*}
  \int_{\Omega \cap B_R(z)} |G(z,z')g(z')|\,dz' 
  &\leq  C \|g\|_{L^\infty(\Omega)} \int_0^R (1 + |\log r| ) r \, dr\\
  & \lesssim R^2 \|g\|_{L^\infty(\Omega)} (|\log R| +1) \\
  & \lesssim \ep^{1+q}\ep^{1-q} |\log \ep|
  \lesssim \ep^{2} |\log \ep| .
\end{align*} 

Outside $B_R(z)$, we use \eqref{eq:aux lemma:G bounds 1} to bound $|G(z,z')|$ by $C(1+|\log R|)\lesssim |\log \ep|$, and hence
\begin{align*}
\int_{\Omega \setminus B_R(z)} |G(z,z')g(z')|\,dz'
&\leq  C|\log \ep| \int_{\Omega \setminus B_R(z)}| g(z')|\,dz' \\
&\leq C |\log \ep| \|g\|_{L^1(\Omega)}
\leq C \ep^{2} |\log \ep|.
\end{align*} 
Thus, we have established \eqref{eq control psi}.  

\end{proof}

\subsection{Technical lemmas}

Recall that for  measurable domains $X \subseteq \mathbb{R}^m$ and $Y \subseteq \mathbb{R}^n$,  the Bochner spaces   $L^p(X; W^{s,p}(Y))$ are defined by
$$
\|f\|_{L^p(X; W^{s,p}(Y))} = \left( \int_X \|f(x, \cdot)\|_{W^{s,p}(Y)}^p \, dx \right)^{1/p}.
$$

We will use the following slicing regularity of  Sobolev functions in the Slobodeckij/Gagliardo scale for $1\le p <\infty $.

\begin{lemma}\label{lemma:Sobolev_slicing}
Let $I$ be an interval and $\Omega = \mathbb{R} \times I$. Then for any $1 \le  p< \infty$ and $s \in [0, \infty)$, there holds
$$ \| f \|_{L^p_x(\mathbb{R}; W^{s,p}_y(I))}  + \| f \|_{L^p_y(I; W^{s,p}_x(\mathbb{R}))}  \lesssim \| f \|_{W^{s,p}(\Omega)}. 
$$
\end{lemma}
\begin{proof}
For $k\in \mathbb N$, the estimate is immediate from Fubini:
\[
\|f\|_{L_x^p(\mathbb R;W_y^{k,p}(I))}^p
=\sum_{j=0}^k \|\p_y^j f\|_{L^p(\Omega)}^p
\lesssim \|f\|_{W^{k,p}(\Omega)}^p,
\]
and similarly
\[
\|f\|_{L_y^p(I;W_x^{k,p}(\mathbb R))}^p
=\sum_{j=0}^k \|\p_x^j f\|_{L^p(\Omega)}^p
\lesssim \|f\|_{W^{k,p}(\Omega)}^p.
\]

Now let $s=m+\theta$ with $m\in \mathbb N_0$ and $0<\theta<1$, and set
\[
X_k:=L_x^p(\mathbb R;W_y^{k,p}(I))
\cap L_y^p(I;W_x^{k,p}(\mathbb R)),
\qquad
\|u\|_{X_k}
:=\|u\|_{L_x^pW_y^{k,p}}+\|u\|_{L_y^pW_x^{k,p}}.
\]
The identity map $T f:=f$ is bounded
\[
T:W^{m,p}(\Omega)\to X_m,
\qquad
T:W^{m+1,p}(\Omega)\to X_{m+1}.
\]
Hence, by real interpolation,
\[
T:(W^{m,p}(\Omega),W^{m+1,p}(\Omega))_{\theta,p}
\longrightarrow
(X_m,X_{m+1})_{\theta,p}.
\]
Since $\Omega=\mathbb R\times I$ and $I$ are Lipschitz extension domains, the
standard interpolation identities yield
\[
(W^{m,p}(\Omega),W^{m+1,p}(\Omega))_{\theta,p}=W^{s,p}(\Omega),
\]
and
\[
(X_m,X_{m+1})_{\theta,p}
=
L_x^p(\mathbb R;W_y^{s,p}(I))
\cap
L_y^p(I;W_x^{s,p}(\mathbb R)),
\]
with equivalent norms. Therefore
\[
\| f \|_{L_x^p(\mathbb R; W_y^{s,p}(I))}
+
\| f \|_{L_y^p(I; W_x^{s,p}(\mathbb R))}
\lesssim
\| f \|_{W^{s,p}(\Omega)}.
\qedhere
\]
\end{proof}

\section{Existence of steady solutions}\label{sec:outline}

In this section, we provide the proofs concerning the flexibility of the 2D Euler equations near the Couette flow on the infinite channel $\Omega = \mathbb{R} \times [-1, 1]$.

We address the flexibility aspect by demonstrating how the existence of smooth, compactly supported relative equilibria (Theorem \ref{thm:flex} and Corollary \ref{cor:flex}) can be reduced to a carefully designed variational maximization problem. The rigorous technical construction and the geometric properties of the maximizer for this variational problem are given in Section~\ref{sec:flex}.

To prove the existence of steady solutions near the Couette flow, we adopt a variational method that relies on a penalized energy maximization problem. The core idea is to balance the non-local self-energy against the interaction energy with the background shear and a specific penalty term to achieve smallness in the maximal possible regularity class.

\begin{definition}[\emph{Vorticity function}]
Fix $0<q \le \frac12$. For any $\ep>0$, we say that $f:\mathbb{R}\to\mathbb{R}$ is a vorticity function if $f$ is smooth, non-negative, and satisfies:
\begin{enumerate}[{\rm(H1)}]
    \item \label{H1}
    $f(t)=0$ for $t \le 0$ and $f$ is strictly increasing on $[0,+\infty)$.

    \item \label{H2}
    $f(t)=\dfrac{\ep^{1-q}}{|\log \ep|^2}\,t \qquad \text{for } t\ge 1.$

    \item \label{H3}
    $|f^{(n)}(t)|\le C_n |\log \ep|^{-2} \ep^{1-q}$. \quad   
\end{enumerate}
\end{definition}

Note that $f$ can be Gevrey-(1+$\delta$) for any $\delta>0$ (Gevrey-$1$ corresponding to analytic functions). That is, there exist constants $K_0, K_1$ depending on $\ep$ such that: $$|f^{(m)}(x)| \le K_0 K_1^m (m!)^{1+\delta}.   $$
We remark that such a function indeed exists. For any fixed $\delta>0$, one can choose a Gevrey-$(1+\delta)$ function $g:\R\to\R$ such that
\[
g(t)=0 \quad \text{for } t\le 0,
\qquad
g(t)=t \quad \text{for } t\ge 1,
\]
and
\[
g(t)>0 \quad \text{and} \quad g'(t)>0 \quad \text{for } t>0.
\]
Then, setting
\[
f(t)=\frac{\ep^{1-q}}{|\log \ep|^2}g(t),
\]
we obtain a function $f$ satisfying \ref{H1}--\ref{H3}.

We first isolate the variational input needed for the proof of Theorem \ref{thm:flex}. The detailed construction of the maximizer and the proof of its properties are deferred to Section \ref{sec:flex}.

\begin{proposition}\label{prop:maximizer} 
 Given $0<q\le \frac12$, there exists $\ep_0>0$ such that for any $0<\ep\leq \ep_0$ and any vorticity function $f:\R \to \R $ satisfying the assumptions \ref{H1}--\ref{H3} with $0<\ep\leq \ep_0$, there exists a  non-negative $\om_*  :\Om \to \R^+ $ satisfying all of the following.

\begin{itemize}
    \item 
      $\om_*$ is a solution to the steady equation:
    \begin{equation}\label{eq:prop:maximizer 1}
    \begin{cases}
      \om_* = f( \ep^{-2} \left(\psi_* - y^2/2 -\alpha ) \right)  &\\
      - \Delta \psi_* = \om_*&
    \end{cases}
    \quad \text{for all $z\in \Omega$}
    \end{equation}
    and the positive constant $\alpha>0$ satisfies $\alpha \gtrsim \ep^2 |\log \ep|$.

    \item $ 0\le \om_* \le  \ep^{1- q}   $ and $\|\om_*\|_{L^1} \le \ep^2$.

    \item $\supp (\om_*) \subset  \{(x,y)\in \Om \mid |x|\le C |\log \ep|^{-1}\}\cap  \{(x,y)\in \Om \mid  |y| \le C \ep |\log \ep|^\frac{1}{2} \}$.

\end{itemize}
    
\end{proposition}
Assuming Proposition \ref{prop:maximizer}, we now derive the steady flexibility result. The traveling-wave version is obtained by the same argument after shifting the background potential; this modification is described in Section \ref{sec:sketch traveling}.

\begin{proof}[Proof of Theorem \ref{thm:flex} assuming Proposition \ref{prop:maximizer}]

Given $\delta>0$ as in Theorem \ref{thm:flex}, we fix $q=\frac{\delta}{4}$ and assume without loss of generality that $\delta \le \frac{1}{100}$. Then we choose a vorticity function $f$ satisfying \ref{H1}--\ref{H3} which is also Gevrey-$(1+q)$.

For this $q>0$ and this $f$, we apply Proposition \ref{prop:maximizer} and obtain for all $0<\ep\le \ep_0$ sufficiently small, a non-negative profile $\om_*$.

Writing $\om:=-\om_*$, it follows from \eqref{eq:prop:maximizer 1} that $\om:\Om \to \R$ is non-positive and satisfies
\begin{equation}\label{eq:proof-thm-flex-EL}
\om  = -f\!\left(\ep^{-2}\bigl(-\psi -y^2/2-\alpha\bigr)\right),
\qquad
-\Delta \psi =\om .
\end{equation}
Note that $u(\om)= \nabla^\perp \psi $, a direct calculation gives
\[
y\p_x\om+u^x(\om)\p_x\om+u^y(\om)\p_y\om=0,
\]
namely $\om$ is a steady solution to \eqref{eq:eu_eq couette}.

Next, we verify the derivative estimates for $\om$. By Proposition \ref{prop:maximizer},
\[
 |\om|\le \ep^{1-q},
\qquad
\|\om\|_{L^1}\le \ep^2
\]
and (set $L:=|\log \ep|$ for simplicity)
\[
\supp (\om)
\subset
\Bigl\{|x|\le C L^{-1}\Bigr\}\cap
\Bigl\{|y|\le C\ep L^{1/2}\Bigr\}.
\]
Let us fix two compact sets $K_0$ and $K_1$ independent of $0<\ep \le \ep_0$,  such that
\[
\supp (\om) \subset K_0 \subset K_1 \subset \Omega,
\qquad
\dist(K_0,\partial K_1)>0 .
\]

\medskip
\noindent
\textbf{Step 1. First and second order derivatives}

For $z\in \supp(\om)$, differentiating \eqref{eq:proof-thm-flex-EL} gives
\begin{equation}\label{eq:proof-thm-flex-grad-omega}
|\nabla \om(z)|
\le
\ep^{-2}\bigl(|\nabla \psi(z)|+|y|\bigr)\,
f'\!\left(\ep^{-2}\bigl(-\psi-y^2/2-\alpha\bigr)\right).
\end{equation}
Using the standard velocity estimate
\begin{equation}\label{eq:proof-thm-flex-grad-psi}
\|\nabla \psi\|_{L^\infty}
\lesssim
\|\om\|_{L^1}^{1/2}\|\om\|_{L^\infty}^{1/2}
\lesssim
\ep^{\frac{3-q}{2}}
\ll \ep,
\end{equation}
together with $|y|\lesssim \ep L^{1/2}$ on $\supp (\om)$ and \ref{H3}, we obtain
\begin{equation}\label{eq:proof-thm-flex-grad-omega-rough}
\|\nabla \om\|_{L^\infty}
\lesssim
\ep^{-2}\cdot \ep L^{1/2}\cdot \bigl(\ep^{1-q}L^{-2}\bigr)
=
\ep^{-q}L^{-3/2}.
\end{equation}
Next, Lemma \ref{lemma:G bounds} gives
\begin{equation}\label{eq:proof-thm-flex-psi-linf}
\|\psi\|_{L^\infty(\Omega)}
\lesssim
\ep^2 L.
\end{equation}
Then by interpolation \footnote{One can estimate the $C^{0,\gamma}$ norm for any fixed Hölder exponent
$0<\gamma \le 1-q$.
Here we take $\gamma=q$ for notational simplicity.} between \eqref{eq:proof-thm-flex-grad-omega-rough} and $\|\om\|_{L^\infty}\le \ep^{1-q}$,
\[
\|\om\|_{C^{0,q}(K_1)}
\lesssim
\|\om\|_{L^\infty}^{1-q}\|\nabla \om\|_{L^\infty}^{q}+\|\om\|_{L^{\infty}}
\lesssim
\ep^{1-2q}L.
\]
Hence the interior Schauder estimate for $-\Delta \psi=\om$ yields
\begin{equation}\label{eq:proof-thm-flex-C2psi}
\|\psi\|_{C^{2,q}(K_0)}
\lesssim
\|\om\|_{C^{0,q}(K_1)}+\|\psi\|_{L^\infty(K_1)}
\lesssim
\ep^{1-2q}L.
\end{equation}

We can now estimate the second derivative of $\om$. Differentiating \eqref{eq:proof-thm-flex-EL} once more gives
\begin{align}
|\nabla^2 \om(z)|
&\lesssim
\ep^{-2}\bigl(|\nabla^2\psi(z)|+1\bigr)
f'\!\left(\ep^{-2}\bigl(-\psi-y^2/2-\alpha\bigr)\right)
\notag\\
&\qquad
+\ep^{-4}\bigl(|\nabla\psi(z)|+|y|\bigr)^2
f''\!\left(\ep^{-2}\bigl(-\psi-y^2/2-\alpha\bigr)\right).
\label{eq:proof-thm-flex-hessian-omega}
\end{align}
Using \eqref{eq:proof-thm-flex-grad-psi}, \eqref{eq:proof-thm-flex-C2psi}, the support bound on $y$, and \ref{H3}, we get
\begin{equation}\label{eq:proof-thm-flex-hessian-omega-rough}
\|\nabla^2 \om\|_{L^\infty}
\lesssim
\ep^{-1-q}L^{-1}.
\end{equation}

\medskip
\noindent
\textbf{Step 2. Higher order derivatives}

We now prove the high-order bounds. The key claim is the following.

\medskip
\noindent\textbf{Claim.}
For every integer $k\ge 0$, there exists $C_k \ge 1$   depending only on $k$, $q$, and $\ep_0$, such that
\begin{equation}\label{eq:proof-thm-flex-claim-log}
\|\nabla^k \om\|_{L^\infty}
\le
C_k\, \ep^{1-q-k} L^{C_k}
\qquad
\text{for all }0<\ep\le \ep_0.
\end{equation}

\medskip
\noindent\emph{Proof of the claim.}
We argue by induction on $k$. Throughout the proof of the claim, $C_k \ge  1$   denotes positive constants depending on the fixed parameters $q,k$ and may change from line to line.

For $k=0$, \eqref{eq:proof-thm-flex-claim-log} is exactly the bound
$
\|\om\|_{L^\infty}\le \ep^{1-q}.
$
For $k=1$, it follows from \eqref{eq:proof-thm-flex-grad-omega-rough}.
For $k=2$, it follows from \eqref{eq:proof-thm-flex-hessian-omega-rough}.
Thus the claim holds for $k=0,1,2$.

Assume now that \eqref{eq:proof-thm-flex-claim-log} holds for all integers $0\le j\le k$, where $k\ge 2$. We prove it for $k+1$.

First, for every $0\le j\le k-1$, the standard interpolation inequality gives
\[
[\nabla^j \om]_{C^q(K_1)}
\lesssim
\|\nabla^j \om\|_{L^\infty}^{1-q}
\|\nabla^{j+1}\om\|_{L^\infty}^{q}.
\]
Using the induction hypothesis, we obtain
\begin{equation}\label{eq:proof-thm-flex-Cjg-omega}
\|\om\|_{C^{j,q}(K_1)}
\le
C_{k}\,\ep^{1-2q-j}L^{ C_k}
\qquad
(0\le j\le k-1).
\end{equation}

Applying the interior Schauder estimate to $-\Delta\psi=\om$, together with \eqref{eq:proof-thm-flex-psi-linf}, we infer that for every $0\le j\le k-1$,
\begin{equation}\label{eq:proof-thm-flex-Cjg-psi}
\|\psi\|_{C^{j+2,q}(K_0)}
\le
C_{k}
\Bigl(
\|\om\|_{C^{j,q}(K_1)}
+\|\psi\|_{L^\infty(\Omega)}
\Bigr)
\le
C_k\,\ep^{1-2q-j}L^{ C_k}.
\end{equation}

Set $\varphi:=\ep^{-2}\bigl(-\psi-y^2/2-\alpha\bigr)$
so that $\om=-f(\varphi)$. We now convert these estimates into bounds for derivatives of $\varphi$.
From \eqref{eq:proof-thm-flex-grad-psi} and the support bound on $y$,
\begin{equation}\label{eq:proof-thm-flex-U1}
\|\nabla \varphi\|_{L^\infty(\supp(\om))}
\lesssim
\ep^{-2}\bigl(\|\nabla\psi\|_{L^\infty}+\|y\|_{L^\infty(\supp(\om))}\bigr)
\lesssim
\ep^{-1}L^{1/2}.
\end{equation}
From \eqref{eq:proof-thm-flex-C2psi},
\begin{equation}\label{eq:proof-thm-flex-U2}
\|\nabla^2 \varphi\|_{L^\infty(\supp(\om))}
\lesssim
\ep^{-2}\bigl(\|\nabla^2\psi\|_{L^\infty(K_0)}+1\bigr)
\lesssim
\ep^{-2}.
\end{equation}
Finally, for every integer $3\le \ell\le k+1$, taking $j=\ell-2$ in \eqref{eq:proof-thm-flex-Cjg-psi} gives
\begin{equation}\label{eq:proof-thm-flex-Ul}
\|\nabla^\ell \varphi\|_{L^\infty(\supp(\om))}
\lesssim
\ep^{-2}\|\nabla^\ell \psi\|_{L^\infty(K_0)}
\le
A_{k}\,\ep^{1-2q-\ell}L^{ C_k}.
\end{equation}

Since $\om=-f(\varphi)$, Faà di Bruno's formula implies that each component of $\nabla^{k+1}\om$ is a finite sum of terms of the form
\begin{equation}\label{eq:proof-thm-flex-fdb-term}
f^{(m)}(\varphi)\prod_{r=1}^{m}\nabla^{\lambda_r}\varphi,
\qquad
m\ge 1,\quad
\lambda_r\ge 1,\quad
\lambda_1+\cdots+\lambda_m=k+1.
\end{equation}
Let $a$ be the number of indices $r$ with $|\lambda_r|=1$ and $b$ be the number of indices with $|\lambda_r|=2$. Then
\[
a+2b+\sum_{|\lambda_r|\ge 3}|\lambda_r|=k+1.
\]
Using \ref{H3}, \eqref{eq:proof-thm-flex-U1}, \eqref{eq:proof-thm-flex-U2}, and \eqref{eq:proof-thm-flex-Ul}, the term \eqref{eq:proof-thm-flex-fdb-term} is bounded by
\begin{align*}
&\lesssim
\bigl(\ep^{1-q}L^{-2}\bigr)
\cdot
\bigl(\ep^{-1}L^{1/2}\bigr)^a
\cdot
\bigl(\ep^{-2}\bigr)^b
\cdot
\prod_{|\lambda_r|\ge 3}
\bigl(\ep^{1-2q-|\lambda_r|}L^{C_k}\bigr)
\\
&\lesssim
\ep^{\,1-q-a-2b+\sum_{|\lambda_r|\ge 3}(1-2q-|\lambda_r|)}
L^{C_k}
\\
&\le
C_k\ep^{\,1-q-(k+1)+c(1-2q)}
L^{C_k},
\end{align*}
where $c$ is the number of indices with $|\lambda_r|\ge 3$.
Since $q \le \frac12$, we have $1-2q \ge0$, which yields
\[
1-q-(k+1)+c(1-2q)\ge 1-q-(k+1).
\]
Hence every term in \eqref{eq:proof-thm-flex-fdb-term} is bounded by
\[
C_k\,\ep^{1-q-(k+1)}L^{C_k}.
\]
Summing over all such terms, for some large constant $C_{k+1}\ge  1$ independent of $0<\ep \le \ep_0$, we obtain
\[
\|\nabla^{k+1}\om\|_{L^\infty}
\le
C_{k+1}\,\ep^{1-q-(k+1)}L^{C_{k+1}}.
\]
This closes the induction and proves the claim.
\hfill$\square$

\medskip

Since $q>0$ is fixed and $0<\ep\le \ep_0$, for each $k$ there exists a constant $C_{k,q,\ep_0}>0$ such that we can absorb the logarithmic factors:
\[
L^{C_k}\le \frac{C_{k,q,\ep_0}}{C_k}\ep^{-q}.
\]
Therefore the claim implies
\begin{equation}\label{eq:proof-thm-flex-final-Linf}
\|\nabla^k \om\|_{L^\infty}
\le
C_{k,q,\ep_0}\,\ep^{1-2q-k}
\qquad
\text{for every integer }k\ge 0.
\end{equation}

\medskip
\noindent
\textbf{Step 3. Sobolev and Gevrey regularity}

We next estimate the Sobolev norms. By the support bound from Proposition \ref{prop:maximizer},
\[
|\supp (\om)|
\lesssim
\ep,
\]
hence for every integer $k\ge 0$ and every $1\le p\le \infty$,
\begin{align}
\|\om\|_{W^{k,p}}
&\lesssim
|\supp (\om)|^{1/p}\sum_{j=0}^k \|\nabla^j\om\|_{L^\infty}
\notag\\
&\lesssim_{k,p,q,\ep_0}
\ep^{1/p}\,\ep^{1-2q-k}
=
\ep^{1-k+\frac1p-2q}.
\label{eq:proof-thm-flex-Wkp}
\end{align}
By interpolation, the same estimate holds for non-integer $s\ge 0$:
\begin{equation}\label{eq:proof-thm-flex-Wsp}
\|\om\|_{W^{s,p}}
\lesssim_{s,p,q,\ep_0}
\ep^{1-s+\frac1p-2q}.
\end{equation}
In particular, for
\[
s=1+\frac1p-\delta,
\]
we obtain
\[
\|\om\|_{W^{s,p}}
\lesssim
\ep^{\delta-2q}
=
\ep^{\delta/2}
\ll 1,
\]
since $q=\delta/4$.

\medskip

Finally, we prove the Gevrey regularity of $\om$.
Recall that  
\[
 \varphi =\ep^{-2}\bigl(-\psi-y^2/2-\alpha\bigr) , 
\]
and by \eqref{eq:proof-thm-flex-EL}  for some   Gevrey-$(1+q)$ function $h$ on $\R$ we have
\[
-\Delta \varphi=h(\varphi) \qquad \text{in $\Omega$}.
\]
By the classical local Gevrey regularity theorem for nonlinear elliptic equations
due to Friedman
\cite{Friedman58,Tri_03}, it follows that $\varphi \in G^{1+q}_{\mathrm{loc}}(\Omega)$. In particular, since $q=\delta/4<\delta$, $\varphi$ is also locally Gevrey-$(1+\delta)$. Therefore,
\[
\ep^2\Delta \varphi+1=\om \in G^{1+q}(\Omega)
\]
since $\om$ has compact support.  

\end{proof}

\section{Variational Construction of the Smooth Steady State}\label{sec:flex}

This section is devoted to the rigorous construction of the smooth steady state required to prove Proposition \ref{prop:maximizer}. We achieve this by maximizing a penalized energy functional over a highly  anisotropic admissible class. Before entering the construction, we give a concise road-map of this section:
\begin{itemize}
    \item We first introduce the variational problem, formulating it as a maximization problem of a penalized energy \eqref{eq:def_penalized energy} in a well-chosen admissible class  \eqref{eq:def_admissible_class};

    \item Then we show the existence (but not uniqueness) of a maximizer $\om_*$ by compactness arguments.  

    \item Next, we show any choice of a maximizer $\om_*$ leads to the desired Euler--Lagrange equation that gives rise to the solution claimed in Proposition \ref{prop:maximizer};

    \item Lastly, we derive useful properties and estimates for the maximizer $\om_*$, thereby proving Proposition \ref{prop:maximizer}.

\end{itemize}

\subsection{The penalized energy and the admissible class}\label{subsec:def_F and A}

Let $f$ be a smooth, non-negative function defined on $\mathbb{R}$ as in  Proposition \ref{prop:maximizer}. To define the penalty term in the variational functional, let $F'$ be the inverse function of $f$ on $[0,+\infty)$ and define
\[
F(x):=\int_0^x F'(s)\,ds.
\]

We seek a non-negative profile $\omega \ge 0$ on the domain $\Omega = \mathbb{R} \times [-1, 1]$ so  that the actual Euler solution will be $-\om$ (note the difference compared to \eqref{eq:E_max_intro}).

The total penalized energy functional is defined as the sum of the non-local self-energy, the interaction energy with the background shear, and the penalty energy:
\begin{equation}\label{eq:def_penalized energy}
\begin{split}
    \mathcal E(\omega)=&  \frac{1}{2} \int_{\Omega }  \mathcal G\omega(z)  \omega(z )\, dz 
     - \frac12 \int_{\Omega } y^2 \omega(z) dz- \ep^{2}\int_{\Omega } F(\omega(z)) dz \\ 
    :=&\mathcal E_1+\mathcal E_2+\mathcal E_3  ,
\end{split}\end{equation}
where $\mathcal G \om(z):= \int_{\Omega } G(z,z') \om(z') \, dz' $.

To capture the anisotropic scaling required to stay below the $W^{1+1/p,p}$ rigidity threshold, we maximize this functional over a specific  admissible class. Fix  $q>0$  small. For any $\ep>0$ we define
\begin{equation}\label{eq:def_admissible_class}
\mathcal A:=\left\{  0\le \omega \le  \ep^{1-q}, \; \int_{\Om}  \omega(z)  dz \leq \ep^{2}, \; \supp(\om ) \subset \{|y| \le \ep^{1-q} \}\right\}. 
\end{equation}

The maximization problem can be formulated as follows:
\begin{problem}\label{problem:maximization_steady}
For all sufficiently small $\ep$, find  $\om_* \in \mathcal{A}     $ such that $\om_* \not \equiv 0$ and $ \mathcal{E}(\om_*) = \sup_{ \om \in \mathcal{A} }  \mathcal{E}(\om )$.
\end{problem}

\begin{remark}
We explain some heuristics behind the design of the penalized energy $\mathcal{E}$ and the admissible class $\mathcal{A}$:
\begin{itemize}
    \item     The specific penalty term $\int F(\om)$ is designed to force the maximizer to have a small $L^\infty$ norm, otherwise, the self-energy tends to concentrate the solution, leading to large Sobolev norms. 
    
    \item In the small $L^\infty$ norm regime, we can eventually show that the $L^\infty$ constraint in $\mathcal{A}$ is never reached, thereby forcing the maximizer to have a smooth profile (in contrast to a patch profile).

 \item  Because of the small $L^\infty$ norm, we can also prove the $y$-support constraint is vacuous, and thus the maximizer satisfies the global Euler--Lagrange equation on $\Omega$.

\end{itemize}
\end{remark}
\subsection{Existence of the maximizer}

To prove that the supremum of $\mathcal{E}$ is attained in $\mathcal{A}$, we first ensure that the maximum energy is strictly bounded away from zero. 
 
The following lemma summarizes some basic properties of $F$.
\begin{lemma}\label{lemma:construction of F}
The function $F$ defined above satisfies the following properties:
\begin{enumerate}
    \item \label{item:F1 bijection}  $F': [0,+\infty) \to [0,+\infty)$ is  bijective with $F''(s) >   0$.
    \item \label{item:F2 sF'-2F}  $sF'(s)-2F(s) \lesssim \frac{\ep^{1-q}}{|\log\ep|^2}$.
    \item \label{item:F3  Fle s} $F(s)\lesssim s$ for $0 \le s \le \frac{\ep^{1-q}}{|\log \ep|^2}$.
    \item \label{item:F4  F' ge log}    $F'(s) = |\log \ep|^2\ep^{q-1}s$ when $s \ge  \frac{ \ep^{1-q}}{  |\log \ep|^2}$.
\end{enumerate}  
All implicit constants above are independent of $\ep$.
\end{lemma}
\begin{proof}
    The definition of $F'$ implies $F''(s)=\frac{1}{f'(F'(s))}> 0$. Next we prove $(2)$.

    When $0 \le s \le \frac{ \ep^{1-q}}{|\log \ep|^2}$, we have
    $$
    0 \le F'(s) \le F'\left(\frac{ \ep^{1-q}}{|\log \ep|^2}\right) = 1
    $$
    since $f(1) =\frac{ \ep^{1-q}}{|\log \ep|^2}$ and $F'$ is the inverse function of $f$. Thus, when $s \le \frac{ \ep^{1-q}}{|\log\ep|^2}$, we have
    $$
   sF'(s)- 2F(s) \le sF'(s) \lesssim s \lesssim \frac{\ep^{1-q}}{|\log\ep|^2}.
    $$

    Now we consider the regime $\frac{  \ep^{1-q}}{|\log \ep|^2} \le s $. On one hand, we have by  definition 
    $$
    sF'(s)=|\log \ep|^2\ep^{q-1}s^2 \qquad \text{when }\frac{  \ep^{1-q}}{|\log \ep|^2} \le s .
    $$
    On the other hand, we can compute and bound $F(s)$ explicitly in this regime:
    \begin{align*}
      F(s)&=F\left(\frac{ \ep^{1-q}}{|\log\ep|^2}\right)+\int_{\frac{ \ep^{1-q}}{|\log\ep|^2}}^s  |\log \ep|^2\ep^{q-1}x \,dx \\
      &=F\left(\frac{ \ep^{1-q}}{|\log\ep|^2}\right)+\frac{|\log\ep|^2\ep^{q-1}s^2}{2} -\frac{\ep^{1-q}}{2|\log \ep|^2} \\
      &\lesssim \frac{|\log\ep|^2\ep^{q-1}s^2}{2}+\frac{\ep^{1-q}}{|\log \ep|^2},
    \end{align*}
    where we have used the fact $F(s) \le sF'(s) \lesssim \frac{\ep^{1-q}}{|\log \ep|^2}$ when $s=\frac{ \ep^{1-q}}{|\log \ep|^2}$ since $F'\left(\frac{ \ep^{1-q}}{|\log \ep|^2}\right) = 1$. Gathering the estimate above, we see that $sF'(s)-2F(s) \lesssim \frac{\ep^{1-q}}{|\log \ep|^2}$.

For $0 \le s \le \frac{\ep^{1-q}}{|\log \ep|^2}$, we have 
$$F(s) \le sF'(s) \le sF'\left(\frac{ \ep^{1-q}}{|\log \ep|^2}\right) =  s,$$ 
which proves $(3)$.
    Finally, $(4)$ follows directly from definition of $F$.  
\end{proof}

\begin{lemma}\label{lemma:positive energy}
Given $\frac12 \ge q>0$, then for any $f$ satisfying \ref{H1}--\ref{H3}, there exists $\ep_0>0$ such that for all $0<\ep\le \ep_0$,
\[
0<\sup_{\omega\in\mathcal A}\mathcal E(\omega)<+\infty.
\]
In particular, any maximizer, once its existence is established, is nontrivial.
\end{lemma}
\begin{proof}
The upper bound follows directly from \eqref{eq control psi}. Then it remains to find one element $\om $ in $\mathcal{A}$ that gives a positive penalized energy $\mathcal{E} >0$.

By testing the functional with a specific trial function of the form  
$$
\om (z)= \frac{\ep^{1-q}}{\pi |\log \ep|^2} \mathbf{1}_{B_\ep}(z),
$$
where the support $B_\ep$ is a ``thin'' ellipse defined by 
$$
B_\ep = \left\{ (x,y) \in \Omega : \left(\frac{x}{\ep^q|\log \ep|^2}\right)^2 + \left(\frac{y}{\ep}\right)^2 \le 1 \right\}.
$$

We compute the total mass of this trial function. The area of the ellipse $B_\ep$ with semi-axes $a = \ep^q|\log \ep|^2$ and $b = \ep$ is $|B_\ep| = \pi a b = \pi \ep^{1+q} |\log \ep|^2$. Thus, the total mass is:
$$
 \int_{\Omega} \om(z) \, dz = \frac{\ep^{1-q}}{|\log \ep|^2} |B_\ep| =  \ep^2.
$$
So the trial function   belongs to the admissible class $\mathcal{A}$.

We now evaluate the total energy $\mathcal{E}(\om) = \mathcal{E}_1(\om) + \mathcal{E}_2(\om) + \mathcal{E}_3(\om)$ term by term.

\vspace{0.5em}
\noindent
\textbf{Step 1: The interaction energy $\mathcal{E}_2$.}

Using absolute bounds on the $y$-scale, we obtain:
$$
\big| \mathcal{E}_2(\om) \big| \leq   \frac{\ep^2}{2} \int_{\Omega}  \om(z) \, dz  \leq   \ep^4.
$$

\vspace{0.5em}
\noindent
\textbf{Step 2: The penalty energy $\mathcal{E}_3$.}

On the support $B_\ep$, the amplitude is strictly $\om(z) = \frac{\ep^{1-q}}{\pi|\log \ep|^2}$. It follows from Lemma \ref{lemma:construction of F} that $F(\om(z)) \lesssim \om(z)$, which implies
$$
\big| \mathcal{E}_3(\om) \big| =  \ep^2 \int_{\Omega} |F(\om(z)) |\, dz \le  \ep^2 \int_{\Omega} \om(z) \, dz \lesssim  \ep^4.
$$

\vspace{0.5em}
\noindent
\textbf{Step 3: The self-energy $\mathcal{E}_1$.}

The Green function $G(z,z') $ has the local asymptotic expansion 
$$
G(z,z') = -\frac{1}{2\pi}\log|z-z'| + H(z,z'),
$$ where the regular part $H(z,z')$ is uniformly bounded on  $z, z' \in \Om \times \Om$.

The maximum distance between any two points $z, z' \in B_\ep$ is bounded by the major axis of the ellipse. Since $0<q <1 $, for  all sufficiently small $\ep>0$, we have $b = \ep \ll a = \ep^q|\log \ep|^2$, and thus, for all $z, z' \in B_\ep$:
$$
|z - z'| \leq 2 \ep^q|\log \ep|^2 \quad \text{for all $z,z' \in \supp(\om) \times \supp(\om)$}.
$$
Again, for all $\ep$ sufficiently small, this implies $-\frac{1}{2\pi}\log|z-z'| + H(z,z') \ge \frac{q}{4\pi}|\log \ep|$.
Therefore, the self-energy is bounded below by:
\begin{align*}
\mathcal{E}_1(\om)  & = \frac{1}{2} \iint_{B_\ep \times B_\ep} G(z,z') \om(z) \om(z') \, dz dz' \\
& \ge \frac{1}{2} \left( \frac{q}{4\pi}|\log \ep| \right) \left( \int_{B_\ep} \om(z) \, dz \right)^2 \\
& \ge \frac{q}{8\pi}|\log \ep|  \ep^4.
\end{align*}

\vspace{0.5em}
\noindent
\textbf{Step 4: Conclusion.}

Summing the three components $\mathcal{E}_i$, the total penalized energy satisfies:
$$
\mathcal{E}(\om) \ge \frac{q}{8\pi} |\log \ep| \ep^4 - C \ep^4 .   
$$
Because $q > 0$ is fixed, the logarithmic term $|\log \ep|$ diverges to $+\infty$ as $\ep \to 0^+$. Thus, we have
$$
\mathcal{E}(\om) \gtrsim |\log \ep|\ep^{4}-\ep^{4}-\ep^4 \gtrsim |\log \ep|\ep^{4} > 0 \qquad \text{for all sufficiently small $\ep$}.
$$
This guarantees that the supremum of the energy in the admissible class $\mathcal{A}$ is strictly positive, hence the class of maximizers is non-empty.
\end{proof}

\subsection{Steiner Symmetrization}
In   variational constructions, Steiner symmetrization is essential to ensure the compactness of maximizing sequences. 

\begin{definition}
We say a non-negative measurable function $\omega$ defined on $\Omega$ has Steiner symmetry if $\omega(x,y)=\omega(-x,y)$ for almost every $(x,y)\in\Omega$,
and for almost every $y$, the function $x\mapsto \omega(x,y)$ is non-increasing for $x\ge 0$.
\end{definition}

The operation of Steiner symmetrization involves replacing the function $x \mapsto \omega(x, y)$ on each horizontal slice with its symmetric-decreasing rearrangement. We formally record its properties in the following proposition.

\begin{proposition}[Steiner symmetrization]\label{Prop Steiner}  
For any non-negative $\omega $ satisfying $\omega\in L^{1}\cap L^{\infty}(\Omega)$ and $y^2\omega\in L^{1}(\Omega)$, there exists a function $\omega_{s} $ with Steiner symmetry such that
\begin{align*}
&||\omega_{s}||_{L^q}=||\omega||_{L^q},\quad 1\leq q\leq \infty,    \\
&||y^2\omega_{s}||_{L^1}=||y^2\omega||_{L^1},   \ \\
&\mathcal E(\omega_{s})\geq \mathcal E(\omega).  
\end{align*}
\end{proposition}
\begin{proof}
First, rearrangements preserve the Lebesgue norms. The symmetrization occurring only in the  $x$-variable also implies the preservation of the $y$-weighted integral. Thus we have $\mathcal{E}_2(\omega_{s}) = \mathcal{E}_2(\omega) $   and $\mathcal{E}_3(\omega_{s}) = \mathcal{E}_3(\omega)$.

 For a proof of $ \mathcal{E}_1(\omega_{s})   \ge  \mathcal{E}_1(\omega)$,  we refer to \cite[Proposition 3.1]{AbeChoi22}.
\end{proof}

We are in a position to construct a maximizer. Let $\{\omega_n\} \subset \mathcal{A}$ be a maximizing sequence for $\mathcal{E}$. Since the sequence $\{\omega_n\}$ is uniformly bounded in $L^2(\Omega)$, we can extract a subsequence, still denoted by $\{\omega_n\}$, and a limit function $\om_* \in \mathcal{A}$ such that $\omega_n \rightharpoonup \om_*$ weakly in $L^2(\Omega)$ as $n \to \infty$.

Note that Lemma \ref{lemma:positive energy} gives $|\mathcal{E}(\om)| \lesssim_{\ep} 1$ for any $\om \in \mathcal A$. To  conclude that $\om_*$ is a maximizer, we must establish that $\lim_{n \to \infty} \mathcal{E}(\omega_n) \le \mathcal{E}(\om_*)$.

\begin{lemma}
Let $\{\omega_n\}$ be a sequence in $\mathcal A$ such that
$
\omega_n\rightharpoonup \omega$ in $ L^{2}(\Omega)
$ as $\ n\to\infty$.
Assume that each $\omega_n$ is Steiner symmetric. Then, 
\begin{equation}\label{eq convergence of E}
\lim_{n\to \infty} \mathcal E(\om_n)\le \mathcal E(\om).
\end{equation}
\end{lemma}
\begin{proof}
Following the argument in \cite[Lemma $3.5$]{AbeChoi22}, we obtain
    \begin{equation}\label{eq convergence of E12}
\lim_{n\to \infty }\mathcal{E}_1[\omega_{n}]= \mathcal{E}_1[\omega] .
\end{equation}
Since $F$ is convex, the functional $\mathcal{E}_3 = \int F(\om)$ is lower semi-continuous. Thus, there hold
$$
 -\mathcal{E}_2[\omega ]=\frac12\int y^2|\om(z)|\,dz \le \liminf_{n \to \infty}\frac12\int y^2|\om_n(z)|\,dz
 $$
and
 $$-\mathcal{E}_3[\omega ]= \ep^2\int F(\om(z))\,dz \le \liminf_{n \to \infty}\ep^2\int F(\om_n(z))\,dz.
$$
Therefore,
$$
\mathcal E_2(\om)+\mathcal E_3(\om) \ge \limsup \mathcal E_2(\om_n)+\mathcal E_3(\om_n),
$$
which completes the proof.
\end{proof}

\mdfsetup{skipabove=2pt,skipbelow=2pt}
\begin{mdframed}[linewidth=1pt,frametitle={The maximizer need not be unique:},nobreak=true]
 
From now on,   we fix one maximizer $\om_*  \in \mathcal A$ and prove Proposition \ref{prop:maximizer} for this $ \om_*$.

\end{mdframed}

\subsection{Amplitude upper bound}

Before deriving the Euler--Lagrange equation, we must ensure that the maximizer does not saturate the upper amplitude constraint of the admissible class $\mathcal{A}$. We prove this via a direct energy comparison, exploiting the rapid growth of the penalty function $F$.

Next, we show that $\om_*$ does not touch the boundary of the constraint $ \om_*(z)\le    \ep^{1-q}$. This is crucial for deriving the desired form of the Euler--Lagrange equation later.
\begin{lemma}\label{lemma:omega upperbound}
There exists $\ep_0>0$ such that for all $0< \ep \leq \ep_0$,   the set $\{z \in \Om \mid \om_*(z)\ge    \ep^{1-q} \}$ has  zero Lebesgue measure.
\end{lemma}
\begin{proof}
Suppose, for contradiction, that there exists a set   $A \subset \R  \times [-\ep^{1-q}, \ep^{1-q}  ] $ of positive measure such that $\om_* |_{A} \ge \ep^{1-q}$.

Due to the constraint $\int \om_*(z)dz \le \ep^{2}$ in the class $ \mathcal{A}$ and the infinite length of the channel, there also exists a set $ B \subset \R  \times [-\ep^{2}, \ep^{2}  ]$ such that $ \om_* |_{B} \le \ep$.

Fix a small parameter $0 <\eta \le \ep $ such that 
\begin{equation}\label{eq: aux lemma:omega upperbound 1}
F'( \ep^{1-q} -\eta) \ge \frac{1}{2} |\log\ep|^2
\end{equation}
This is possible thanks to Lemma \ref{lemma:construction of F}. Using the chosen sets $A,B$ and this parameter $\eta>0$, we define
    $$
    \om(z)=\om_*(z)-\eta\mathbf{1}_{A}+\eta\mathbf{1}_{B},
    $$
By shrinking the set  $A$ or $B$ if necessary, we can assume  $0<|A| =|B| \le \ep $, and hence $\om \in \mathcal{A}$ by design. 

We can now compute the difference of penalized energy between the maximizer $ \om_*$ and the perturbed $\om$:
\begin{align*}
        \mathcal{E}(\om)-\mathcal{E}(\om_*)=& \Delta_1 + \Delta_2 + \Delta_3
     \end{align*} 
     where $\Delta_1=\mathcal{E}_1(\om)-\mathcal{E}_1(\om_*)$ includes all the bilinear terms:
     \begin{equation}
      \begin{aligned}
          \Delta_1 &   =  -  \eta\int \mathcal{G} \om_*(z) \mathbf{1}_{A}  \,dz + \eta\int \mathcal{G}  \om_*(z) \mathbf{1}_{B} \,dz \\ 
        &\quad - \eta   \int \mathcal{G}\eta \mathbf{1}_{B} (z)\mathbf{1}_{A}(z )\,dz \\
        &\quad + \frac{\eta }{2}  \int \mathcal{G}\eta\mathbf{1}_{B}(z)\mathbf{1}_{B}(z)\,dz+ \frac{\eta }{2} \int \mathcal{G}\eta \mathbf{1}_{A}(z)\mathbf{1}_{A}(z)\,dz,
      \end{aligned}
  \end{equation}      
   $\Delta_2=\mathcal{E}_2(\om)-\mathcal{E}_2(\om_*)$ consists of the linear changes:
  \begin{equation}
      \begin{aligned}
         \Delta_2 =  \eta\int y^2\mathbf{1}_{A}(z)\,dz -\eta\int y^2 \mathbf{1}_{B}(z)\,dz ,
      \end{aligned}
  \end{equation}   
  and $\Delta_3=\mathcal{E}_3(\om)-\mathcal{E}_3(\om_*)$ is the change of penalty energy:
     \begin{equation}
      \begin{aligned}
  \Delta_3 = \ep^2\int F(\om_*)\,dz -\ep^2\int F(\om)\,dz.
      \end{aligned}
  \end{equation}

The choice $|A|,|B|\le \ep$ and $\eta \le \ep$ implies that $ \eta 1_{B}, \eta 1_A$ are all in the admissible class $\mathcal{A}$ and hence Lemma \ref{lemma:G bounds} gives 
\begin{align}
   | \Delta_1 | & \lesssim \eta \ep^2 |\log \ep| \Big(  \int 1_A + \int 1_B \Big)  \nonumber  \\
   & \lesssim \eta \ep^2 |\log \ep| |A| . \label{eq: aux lemma:omega upperbound 2}
\end{align}

For the linear change, we estimate its lower bound by dropping the positive term, 
 \begin{equation}\label{eq: aux lemma:omega upperbound 3}
        \Delta_2  \geq       - \eta\int y^2 \mathbf{1}_{B}(z)\,dz = -\eta |B| \ep^{4} = -\eta |A| \ep^{4} . 
 \end{equation}

Finally, for the penalty energy,
\begin{align*}
 \Delta_3=  \ep^2 \int_{\Om} F(\om_*)\,dz -\ep^2 \int_{\Om} F(\om)\,dz
 =& \ep^2 \int_{B} F(\om_*)dz-\ep^2\int_{B} F(\om_*+ \eta )dz\\
 &+\ep^2\int_{A} F(\om_*)dz-\ep^2\int_{A} F(\om_*-\eta)dz.
    \end{align*}
We first consider the change on the set $B$. Recall that  $\om_* |_{A} \ge \ep^{1-q}$ and $ \om_* |_{B} \le \ep$. From this, the convexity of $F$ implies that
\begin{equation}\label{eq: aux lemma:omega upperbound 4a}
\ep^2 \int_{B} F(\om_*)dz-\ep^2\int_{B} F(\om_*+ \eta )dz \ge   - \ep^2  \eta \int_{B} F'(2 \ep)dz .
\end{equation}
By taking $\ep_0>0$ sufficiently small, $2\ep \leq   \frac{\ep^{1-q}}{|\log \ep|^2}$, then the monotonicity of $F'$ together with \eqref{item:F4  F' ge log}  of Lemma \ref{lemma:construction of F} yield 
$$
\int_{B} F'(2 \ep)dz \le \int_{B} F'(\frac{\ep^{1-q}}{|\log \ep|^2})dz  = |B|  =|A| ,
$$ 
and hence by \eqref{eq: aux lemma:omega upperbound 4a} we get
\begin{equation}\label{eq: aux lemma:omega upperbound 4aa}
 \ep^2 \int_{B} F(\om_*)dz-\ep^2\int_{B}F(\om_*+ \eta )dz     \ge  - \eta |A|  \ep^{2 }  .
\end{equation}

Finally, we show the gain term of removing mass from the set $A$ outweighs all the other losses. Indeed, the same convexity argument using \eqref{eq: aux lemma:omega upperbound 1} implies
\begin{equation}\label{eq: aux lemma:omega upperbound 4b}
\begin{aligned}
\ep^2 \int_{A} F(\om_*)dz-\ep^2\int_{A} F(\om_* - \eta )dz & \ge     \ep^2 \eta \int_{A} F'(\om_* - \eta )dz \\
&\ge  \frac{\ep^2 \eta}{2} \int_{A}   |\log \ep|^{ 2} dz \\
& \ge \frac{\eta   |A|  \ep^2}{2}  |\log \ep|^{ 2}.
\end{aligned}
\end{equation}

Gathering the estimates \eqref{eq: aux lemma:omega upperbound 2}, \eqref{eq: aux lemma:omega upperbound 3}, \eqref{eq: aux lemma:omega upperbound 4aa} and \eqref{eq: aux lemma:omega upperbound 4b}, we see that the total change of the penalized energy is positive:
\begin{align*}
      \mathcal{E}(\om)-\mathcal{E}(\om_*) & =  \Delta_1 + \Delta_2 + \Delta_3 \\
      & \geq   \eta  |A|  \ep^{2} |\log \ep|^{ 2}-  C \eta |A|  \ep^2 |\log \ep|  - C  \eta |A| \ep^{2}  - C \eta|A|  \ep^4  .
\end{align*}

By taking $\ep_0>0$   sufficiently small,  for all $0<\ep \le \ep_0$, the positive first term dominates all other negative terms above, and we have
 \begin{align*}
      \mathcal{E}(\om)-\mathcal{E}(\om_*)  
        \geq  \frac{1}{2} \eta  |A|  \ep^{2} |\log \ep|^{ 2} 
\end{align*}
which is a contradiction to the maximality of $\mathcal{E}(\om_*)$.

\end{proof}

\subsection{The Euler--Lagrange equation}

With the upper bound of the maximizer $\om_* \in \mathcal{A}$ established, we can compute the first variation of the penalized energy functional that leads to the Euler--Lagrange equation.

We define the effective potential associated with the maximizer $\om_*$ as:
\begin{equation}\label{eq:def_H}
H(z) := \psi_*(z) - \frac{1}{2}y^2 - \ep^2 F'(\om_*(z))
\end{equation}
where $\psi_* = \mathcal{G}\om_*$ is the stream function.

\begin{proposition}\label{prop:Lagrange multiplier}
Let $\om_* \in \mathcal{A}$ be a maximizer of the energy $\mathcal{E}$. There exists a constant $\alpha \in \mathbb{R}$ (the Lagrange multiplier) such that  the effective potential $H$   given by \eqref{eq:def_H} satisfies
\begin{itemize}
\item for a.e.  $z \in \supp(\om_*)$, $H(z) = \alpha$.
    \item for a.e. $z \notin \supp (\om_*)$, $H(z) \le \alpha$. 
\end{itemize}

Moreover, for all sufficiently small $\ep>0$, the multiplier $\alpha \in \R^+$ is positive and satisfies the lower bound
\begin{equation}
    \alpha \gtrsim \ep^2 |\log \ep|.
\end{equation}
\end{proposition}
\begin{proof}

\hfill

\noindent
\textbf{Step 1: Constancy of $H$:}

First of all, the admissible class $\mathcal{A}$ is convex: if $f,g \in \mathcal{A}$, so is $ tf + (1-t) g$ for any $ 0\le t\le 1$.

Since  $\om_*$ is a maximizer in $\mathcal{A}$, we must have $ \mathcal{E}(\om_*+ t(g - \om_*) ) -  \mathcal{E}(\om_* ) \le 0$ for any $ 0\le t \le 1$ and any $g\in \mathcal{A}$.

By taking $t\to 0+$, we have
\begin{equation}\label{eq:global_var}
\int_{\Omega } H(z) (g(z) - \om_*(z)) \, dz \le 0 \quad \text{for all } g \in \mathcal{A}.
\end{equation}

If $H(z)$ is not a constant on the full $\supp (\om_*)$, a suitable choice of $g$ then implies a contradiction to \eqref{eq:global_var}. Indeed, there exist  two sets $A,B  \subset  \supp (\om_*) $  of positive measure on which $H(z)$ takes different values, say $ H|_{B} > H |_{A}$.

Since $\ep^{1-q}>\om_*> 0$ $a.e.$ on $\supp (\om_*)$ by Lemma \ref{lemma:omega upperbound}, after
possibly replacing \(A\) and \(B\) by subsets of equal positive measure, we may
choose \(\eta>0\) sufficiently small such that $\om_* \ge \eta$ in $A$ and $\om_* \le \ep^{1-q}-\eta$ in $B$, which gives
\[
g:=\omega_\ast-\eta 1_A+\eta 1_B\in \mathcal A .
\]
Then we have a  contradiction: 
\begin{align*}
\int_{\Omega} H(z) (g(z) - \om_*(z)) \, dz &= \int_{\Omega} H(z) (-\eta \mathbf{1}_{A} + \eta \mathbf{1}_{B}) \, dz \\
& = \eta \int_{B} H(z) \, dz - \eta \int_{A} H(z) \, dz > 0.
\end{align*} 
 
Let $ S_\ep =\{ |y| \le \ep^{1-q}\}$ be the strip from the class $\mathcal{A}$. Assuming there exists a set $A \subset S_{\ep}$ outside $\supp (\om_*)$ such that $H(z) >\alpha$ for $z\in A$, then arguing as above using \eqref{eq:global_var},  we get 
\begin{equation}\label{eq H(z) in Sep}
    H(z) \le \alpha \quad \text{a.e. on $S_\ep \setminus \supp (\om_*)$}.
\end{equation}
Note that $|y| \ge \ep^{1-q}$ on $\Omega \setminus S_{\ep}$, it follows from Lemma \ref{lemma:G bounds} that $H(z)\leq C \ep^2 |\log \ep| - \ep^{2-2q}<0$ when $z \in \Om\setminus S_{\ep}$.  Therefore, together with \eqref{eq H(z) in Sep}, we obtain
\begin{equation}\label{eq H(z) in Om}
    H(z) \le \max\{\alpha\,,0\} \quad \text{a.e. on $\Om \setminus \supp (\om_*)$}.
\end{equation}

\vspace{0.5em}
\noindent
\textbf{Step 2: Estimate of $\alpha$:}

Recall the energy functional in terms of $\om_*$,
    $$
    \mathcal E(\om_*) = \frac{1}{2} \int \psi_* \om_* - \frac{1}{2} \int y^2 \om_* - \ep^{2} \int F(\om_*),
    $$
and the Euler--Lagrange equation on the support of $\om_*$, \begin{equation}\label{eq:EL_corrected}
        \psi_* - y^2/2 - \alpha  = \ep^{2} F'(\om_*).
    \end{equation}

    To bound $\alpha$, multiplying \eqref{eq:EL_corrected} by $\frac{1}{2}\om_*$ and integrating over the domain $\Omega$:
    $$
    \frac{1}{2} \int \psi_* \om_* - \frac{1}{4} \int y^2 \om_* - \frac{1}{2} \alpha \int \om_* = \frac{1}{2} \ep^{2} \int \om_* F'(\om_*).
    $$
    Notice that we can reconstruct $E(\om_*)$ on the left-hand side:
    $$
    E(\om_*) + \frac{1}{4} \int y^2 \om_* + \ep^{2} \int F(\om_*) - \frac{1}{2} \alpha \int \om_* = \frac{1}{2} \ep^{2} \int \om_* F'(\om_*).
    $$
    Since $\int y^2 \om_* \ge 0$, we obtain the inequality
    \begin{equation}\label{eq: aux prop EL multipler 1}
       E(\om_*) - \frac{1}{2} \alpha \int \om_* \le \frac{1}{2} \ep^{2} \int_{\Om} \left( \om_* F'(\om_*) - 2 F(\om_*)\right).
    \end{equation}
    We first prove the proposition assuming the following.
    
    \noindent
    \textbf{Claim:} for any $\ep>0$, there holds the estimate
    \begin{equation}\label{eq claim wF'(w)-2F(w)}
     \int_{\Om} \om_* F'(\om_*) - 2 F(\om_*)  \lesssim \ep^2.
    \end{equation}

    Using the \textbf{Claim} in \eqref{eq: aux prop EL multipler 1}, we then  have
    $$
    \alpha \int \om_* \gtrsim E(\om_*)-C\ep^{4}>0.
    $$ 
    Since the maximal energy $E(\om_*) \gtrsim |\log \ep| \ep^{4}$, solving for $\alpha$ gives:
    \begin{equation}\label{eq:alpha_global}
        \alpha \gtrsim   \ep^2|\log \ep|,
    \end{equation}
which, together with \eqref{eq H(z) in Om} yields
\begin{equation}
    H(x) \le \alpha \quad \text{a.e. on $\Om \setminus \supp (\om_*)$}.
\end{equation}

    Now it remains to prove the claim \eqref{eq claim wF'(w)-2F(w)}. We decompose the integral into
    \begin{align*}
        \int_{\Om} \om_* F'(\om_*) - 2 F(\om_*)&=\int_{\om_* \le \frac{\ep^{1-q}}{|\log \ep|^2}}   +\int_{\om_*\ge \frac{\ep^{1-q}}{|\log \ep|^2}}  .
    \end{align*}
    For the first part, using   Lemma \ref{lemma:construction of F}, we get
    $$
    \int_{\om_* \le \frac{\ep^{1-q}}{|\log \ep|^2}} \om_* F'(\om_*) - 2 F(\om_*) \lesssim \int_{\om_* \le \frac{\ep^{1-q}}{|\log \ep|^2}} \om_* F'(\om_*)\lesssim \int_{\om_* \le \frac{\ep^{1-q}}{|\log \ep|^2}} \om_* \lesssim \ep^2 .
    $$
    Meanwhile, for the second part, since $\int \om_*\le\ep^2$, we obtain $$
    \left| \supp\left(  \textbf{1}_{\om_* \ge \frac{\ep^{1-q}}{|\log \ep|^2}}\right) \right| \lesssim \ep^{1+q}|\log \ep|^2,
    $$
    which, together with \eqref{item:F2 sF'-2F} of Lemma \ref{lemma:construction of F} yields
    $$\int_{ \om_* \ge \frac{\ep^{1-q}}{|\log \ep|^2}} \om_* F'(\om_*) - 2 F(\om_*) \lesssim \ep^2.$$
    
    This completes the proof of the claim.

\end{proof}

\subsection{Bounds on the Support}
To finalize the proof of Proposition \ref{prop:maximizer}, we must establish the  geometric properties of the maximizer,   bounding its anisotropic support.

First, we need an auxiliary lemma that shows the horizontal decay of the stream function.
\begin{lemma}\label{lemma:psi decay by steiner}
    Assuming $\rho \in L^1(\Om)$ has Steiner symmetry, then there holds 
    $$
    |\mathcal{G} \rho(z)| \lesssim \frac{\|\rho\|_{L^{1}}}{|x|}.
    $$
\end{lemma}
\begin{proof}
    Set $M=\|\rho\|_{L^1}$ and 
    $$
    m(x)=\int_{-1}^1 |\rho(x,y)|\,dy.
    $$ 
    The Steiner symmetry of $\rho$ then implies that $m(x)$ is non-increasing when $x \ge 0$. It follows that for any $x\ge 0$,
    \begin{equation}
        \frac{xm(x)}{2}\le \int_{x/2}^x m(s)\,ds \le\int_0^{\infty} m(s)\,ds =\frac{M}{2}
    \end{equation}
    and hence  
    \begin{equation}\label{eq:aux lemma:psi decay by steiner 1}
    m(x)=m(|x|) \le \frac{M}{|x|} \quad \text{for any $x\in \R$}.
   \end{equation}
   Now for the integral $\mathcal{G} \rho$, we first bound the integrand by its absolute value
    $$|\mathcal{G} \rho (z)| \lesssim \iint |G(z,z')| |\rho(x',y')|\,dx'\,dy',$$
    and then decompose the integral domain into the region $|x'| \le |x|/2$ and $|x'| \ge |x|/2$.

    \noindent
    \textbf{Case I: $|x'| \le |x|/2$}. In this case, we have $|z-z'|\ge |x-x'| \ge  \frac{|x|}{2}$. Then Lemma \ref{lemma:kernel G} gives
    $$
    |G(z,z')| \leq  \frac{1}{|x-x'|} \lesssim \frac{1}{|x|},
    $$ 
    which implies
    $$
    \iint_{|x'| \le |x|/2} |G(z,z')| |\rho(x',y')| \,dx'\,dy' \lesssim \frac{ \|\rho\|_{L^1}}{|x|} \lesssim \frac{M}{|x|}.
    $$

    \noindent
    \textbf{Case II: $|x'| \ge |x|/2$}. In this case, \eqref{eq:aux lemma:psi decay by steiner 1} yields
    \begin{equation*}
        \begin{aligned}
            \iint_{|x'| \ge |x|/2} |G(z,z')| |\rho(x',y') |\,dx'\,dy' &\lesssim \iint_{|x'| \ge |x|/2} |G(z,z')| |\rho(\frac{x}{2},y') | \,dx'\,dy' \\
            & \lesssim \int_{-1}^1 |\rho(\frac{x}{2},y')|\,dy' \\
            &=m(\frac{x}{2}) \lesssim \frac{M}{|x|},
        \end{aligned}
    \end{equation*}
    which completes the proof.
\end{proof}

Applying Proposition \ref{Prop Steiner}, we may assume that our maximizer $\om_*$ is Steiner symmetric. This allows us to use the Euler--Lagrange equation and the above lemma to estimate the size of $\supp (\om_{*})$.

\begin{lemma}\label{lemma:xy_support} 
    There exists $\ep_0 >0$ such that for any $0<\ep \le \ep_0$, the maximizer $\om_*$ is supported in the region $|x| \lesssim  {|\log \ep|}^{-1} $ and $|y| \lesssim \ep|\log \ep|^\frac{1}{2} $
\end{lemma}
\begin{proof}
 
 \hfill

 \noindent
 \textbf{Part 1: the $x$-support.} Recall the Euler--Lagrange equation on the support of $\om_*$:
 \begin{equation}\label{eq:aux lemma:xy_support 0}
        \psi_* - y^2/2 - \alpha  = \ep^{2} F'(\om_*).
    \end{equation}
    Since  both $F'(\om_*)$ and $y^2/2$  are non-negative  on $ \supp(\om_*)$,  for any point $z=(x,y) \in \supp(\om_*)$, we have 
  \begin{equation}\label{eq:aux lemma:xy_support 1}
    \psi_*(x,y) \ge y^2/2 + \alpha \ge \alpha  \gtrsim |\log \ep| \ep^{2},  
  \end{equation}
 where we have also used the lower bound on $\alpha$ from   Proposition \ref{prop:Lagrange multiplier}.

    Finally, by the Steiner symmetry  Lemma \ref{lemma:psi decay by steiner}, on $ \supp(\om_*)$
    $$
    \psi_*(x,y) \lesssim \frac{\| \om_* \|_{L^1}}{|x|} \le \frac{\ep^{2}}{|x|}.
    $$
    Combining this with our lower bound \eqref{eq:aux lemma:xy_support 1} yields:
    $$
    |\log \ep| \ep^{2} \lesssim \frac{\ep^{2}}{|x|} \implies |x| \lesssim \frac{1}{|\log \ep|}.
    $$
    \noindent
 \textbf{Part 2: the $y$-support}

Using the lower bound on the multiplier $\alpha$ from Proposition  \ref{prop:Lagrange multiplier} again, we have on $\supp (\om_*)$
\begin{equation}
\psi_*(x,y) - y^2/2 - \ep^2 F'(\om_*)   \ge C |\log \ep| \ep^{2} .
\end{equation}
Since $F'\ge 0$, this means that
\begin{equation}
 y^2  \leq 2\psi_*(x,y) \quad \text{for all $z =(x,y) \in \supp (\om_*)$}.
\end{equation}
Thus by Lemma \ref{lemma:G bounds} again, $|y| \le C |\log \ep|^\frac{1}{2} \ep $ for any $ z=(x,y) \in \supp (\om_*)$.

\end{proof}

\subsection{Conclusion of the proof}

We now verify that all the claims in Proposition \ref{prop:maximizer} have been established. It remains to show the semilinear equation \eqref{eq:prop:maximizer 1} on the full channel $\Omega$.

By Proposition \ref{prop:Lagrange multiplier}, on the support of $\om_*$ we have
\begin{equation}
    H(z) = \psi_* - \frac{y^2}{2} -\ep^2 F'(\om_*(z)) = \alpha.
\end{equation}
So  the equation \eqref{eq:prop:maximizer 1} is satisfied on $\supp(\om_*)$. 

It remains to show \eqref{eq:prop:maximizer 1} on $\Om \setminus \supp (\om_*)$. Thanks to Proposition \ref{prop:Lagrange multiplier} again, on $\Om \setminus \supp (\om_*)$, there holds $H(z) \le \alpha$, i.e.   $\psi_* - \frac{y^2}{2} -\alpha \le 0 $. Since the vorticity function $f(t ) = 0$ for $t \le 0$, we have  $ \om_* = f( \ep^{-2}  (\psi_* - y^2/2 -\alpha ) )$ as desired.

\section{Traveling waves: modification of the variational construction}\label{sec:sketch traveling}

We now explain how the variational construction from Section \ref{sec:flex} is modified to produce the traveling waves in Corollary \ref{cor:flex}.

\subsection{The modified setup for traveling waves}\label{subsec:traveling def_F and A}

The only change is that the vorticity is centered near the critical layer \(y=c\) rather than near \(y=0\). Accordingly, we keep the same class of vorticity functions as in Section \ref{sec:flex}:

\begin{enumerate}
     \item $f(0)=0$ and $f$ is strictly increasing on $[0,+\infty)$. 
     \item $f(t) =  \frac{\ep^{1-q} }{|\log \ep|^2} t $ when $  t \ge 1 $.  
     \item $|f^{(n)}(t)|  \le C_n |\log \ep|^{-2}\ep^{1-q}$.
 \end{enumerate}
and the same penalty function:
$$
F(t) := \int_0^t F'( \tau ) \,d\tau ,
$$ 
where   $F'$ is also the inverse function of $f$ on $[0, +\infty)$. 

As in the steady case, we maximize over non-negative profiles. The interaction with the Couette flow is now measured relative to the level \(y=c\), so the penalized energy and admissible class become 
\begin{equation}\begin{split}
    \mathcal E(\omega)=&  \frac12\int_{\Omega }  \mathcal G\omega(z)  \omega(z )\, dz 
     - \frac12 \int_{\Omega } (y-c)^2 \omega(z) dz- \ep^{2}\int_{\Omega } F(\omega(z)) dz  .
\end{split}
\end{equation}
Fix two parameters: a  small exponent $0<q\le \frac12$ and a traveling speed $c\in(-1,1)$. For any $\ep>0$, define the admissible class
\begin{equation}
\mathcal A:=\left\{  0\le \omega \le  \ep^{1-q}, \; \int_{\Om}  \omega(z)  dz \leq \ep^{2}, \; \supp(\om) \subset \{|y-c| \le \ep^{1-q} \}\right\}. 
\end{equation}

With these shifted quantities, the traveling-wave variational problem takes the same form as before:
 \begin{problem}\label{problem:maximization_travel}
Given $\frac12 \ge q>0$ and $c\in(-1,1)$, for all sufficiently small $\ep>0$, find $\om_* \in \mathcal{A}     $ such that $\om_* \not \equiv 0$ and $ \mathcal{E}(\om_*) = \sup_{ \om \in \mathcal{A} }  \mathcal{E}(\om )$.
\end{problem}

\begin{proposition}\label{prop:maximizer_with_c} 
 Given any $f$ satisfying the assumptions and any exponent $\frac12 \ge q>0$, there exists $\ep_0>0$ such that for any $0<\ep\leq \ep_0$ there exists a  non-negative $\om_*  :\Om \to \R^+ $ satisfying all of the following.

\begin{itemize}
    \item 
     $\om_*$ is a solution to the steady equation:
    \begin{equation}
    \begin{cases}
     \om_* = f( \ep^{-2} \left(\psi_* - (y-c)^2/2 -\alpha ) \right), &\\
       -\Delta \psi_* = \om_* &
    \end{cases}
    \end{equation}
    and the positive constant $\alpha$ satisfies $\alpha \gtrsim \ep^2 |\log \ep|$.

    \item $ 0\le \om_* \le  \ep^{1- q}   $ and $\|\om_*\|_{L^1} \le \ep^2$.

    \item $\supp (\om_*) \subset  \{ |x|\le C |\log \ep|^{-1}\}\cap  \{ |y-c| \le C \ep |\log \ep|^\frac{1}{2}\}$.

\end{itemize}
    
\end{proposition}

\subsection{Sketch of the proof}
We indicate why no new ingredient is needed. Each step of the proof in Section \ref{sec:flex} is stable under the shift from \(y\) to \(y-c\):
\begin{enumerate}

    \item The test configuration in Lemma \ref{lemma:positive energy} used to make the energy positive is simply shifted at height $y=c$, so the lower bound on the supremum is identical.

    \item Steiner symmetrization still acts only in the $x$-variable, hence it is unaffected by the shift in $y$.

    \item The compactness argument for maximizing sequences, the derivation of the Euler--Lagrange equation, and the lower bound for the multiplier $\alpha$ are the same as in the steady case.

\end{enumerate}

Proposition \ref{prop:maximizer_with_c} therefore yields Corollary \ref{cor:flex} by the same regularity and smallness estimates used in the proof of Theorem \ref{thm:flex}.

\section{Rigidity: non-existence}\label{sec:rigid}
We now turn to the rigidity side of the dichotomy. We prove that sufficiently small traveling-wave perturbations of the Couette flow must vanish once their vorticity lies above the critical threshold, namely in \(W^{s,p}\) with \(s>1+\frac1p\), or in \(C^{1+}\).

The key point is the singular behavior near the critical layer, where the horizontal transport speed \(y+c+u^x\) may vanish. Rewriting the transport equation gives
\[
\p_x\omega = - \frac{u^y \p_y\omega}{y+c+u^x}.
\]

\subsection{Non-existence in the Sobolev scale}
We begin with the Sobolev rigidity theorem. The proof relies on three elementary estimates near the critical layer. These estimates allow us to control the singular factor above in an elliptic energy estimate for \(u^y\), which then forces \(u^y\), and hence the vorticity, to vanish.

The first lemma concerns the smoothness of  1D  functions  using Hardy-type estimates. 
\begin{lemma}\label{lem:1d-zero-maximal}
Let $I=[-1,1]$ and let $h\in H^1(I)$. Assume that $h(y_0)=0$ for some $y_0\in I$. Then
\[
|h(y)| \lesssim |y-y_0|\, \mathcal M_y(\p_y h)(y)
\qquad \text{for a.e. }y\in I,
\]
where $\mathcal M_y$ denotes the Hardy--Littlewood maximal operator in the $y$-variable.
Consequently, for every $1<p\le \infty$,  
\[
\bigl\||y-y_0|^{-1} h\bigr\|_{L^p_y(I)}
\lesssim \|\p_y h\|_{L^p_y(I)}.
\]
\end{lemma}
\begin{proof}
For $y\in I$, by the fundamental theorem of calculus,
\[
h(y)=\int_{y_0}^y \p_y h(\tau)\,d\tau.
\]
Hence
\[
|h(y)| \lesssim |y-y_0|\, \mathcal M_y(\p_y h)(y).
\]
The $L^p$ estimate follows from the $L^p$ boundedness of $\mathcal M_y$ for $1<p\le \infty$.
\end{proof}

Our analysis near the singular layer also relies on the following mixed-index embedding result.
\begin{lemma}\label{lem:mixed-norm-H1}
Let $\Omega=\R\times[-1,1]$. For any $p>2$ and $q<\infty$, there exists a constant $C=C(p,q)>0$ such that
\[
\|v\|_{L_x^{p}L_y^{q}}\le C \|v\|_{H^1(\Omega)}
\]
for all $v\in H^1(\Omega)$.
\end{lemma}
\begin{proof} 
Since $y\in [-1,1]$ is bounded, by H\"older's inequality in $y$, it suffices to consider $q > \max\{p,2\}$ such that $\sigma:= \frac{1}{p} - \frac{1}{q}>0 $. By 1D Sobolev embedding
\begin{equation}
    \|v\|_{L_x^pL_y^q(\Omega)} \lesssim \|v\|_{L_x^p W_y^{\sigma,p}(\Omega)}.
\end{equation}
Then by Lemma \ref{lemma:Sobolev_slicing}, the right-hand side is controlled by $ \|v\|_{ W^{\sigma,p}(\Omega)} $. Since $p>2 $, the result  follows immediately by the 2D Sobolev embedding (omitting $\Omega$ for brevity)
\begin{equation}
\|v\|_{ W^{\sigma,p} }\lesssim \|v\|_{ H^{\sigma+2(\frac{1}{2} -\frac{1}{p})} } = \|v\|_{ H^{1-\frac{1}{p} -\frac{1}{q}} } \le \|v\|_{ H^{1}}
\end{equation}
thanks to $ \sigma= \frac{1}{p} - \frac{1}{q}>0$.

\end{proof}

The previous lemma supplies the mixed \(L_x^pL_y^q\) control needed for the energy method.   The next lemma handles the singular weight \(|y-y_*(x)|^{-1}\) appearing near the critical layer, using the  fractional Sobolev regularity of $\om$.

\begin{lemma}\label{lem:weighted-critical-layer-sobolev}
Let  $s>1+\frac1p$. There exists $ C_{s,p} >0$ such that the following holds.
Let $\phi\in W^{s,p}(\Omega)$ and let $y_*=y_*(x)\in [-1,1]$ be any measurable function of $x$.
Assume that for a.e. $x$ in a measurable set $B\subset \R$ \footnote{This is well-defined. Indeed Lemma \ref{lemma:Sobolev_slicing} implies that $\partial_y\phi(x,\cdot)$ is a continuous function for a.e. $x \in R$.},
\[
\p_y \phi(x,y_*(x))=0.
\]
Then for every $v\in H^1(\Omega)$,
\[
\int_B \int_{-1}^1
|v(x,y)|^2 \frac{|\p_y \phi(x,y)|}{|y-y_*(x)|}\,dy\,dx
\le C_{s,p} \|\phi\|_{W^{s,p}(\Omega)} \|v\|_{H^1(\Omega)}^2.
\]
\end{lemma}
\begin{proof}
For each fixed $x\in B$, define
\[
g_x(y):=\p_y\phi(x,y), \qquad y\in(-1,1).
\]
By assumption,
\[
g_x\bigl(y_*(x)\bigr)=0 \quad \text{for a.e. $x\in B$.}
\]

Since $\phi\in W^{s,p}(\Omega)$ with $s>1+\frac1p$, we have for a.e. $x\in B$ that
\[
g_x\in W^{s-1,p}(-1,1),
\qquad s-1>\frac1p.
\]
Hence, by the 1D Sobolev embedding
\[
W^{s-1,p}(-1,1)\hookrightarrow C^{0,s-1-\frac1p}([-1,1]),
\]
we obtain
\[
|g_x(y)-g_x(y_*(x))|
\lesssim
|y-y_*(x)|^{s-1-\frac1p}\,
\|g_x\|_{W^{s-1,p}(-1,1)}.
\]
Since $g_x(y_*(x))=0$, it follows that
\[
|g_x(y)|
\lesssim
|y-y_*(x)|^{s-1-\frac1p}\,
\|g_x\|_{W^{s-1,p}(-1,1)}.
\]
Therefore,
\[
\frac{|\p_y\phi(x,y)|}{|y-y_*(x)|}
=
\frac{|g_x(y)|}{|y-y_*(x)|}
\lesssim
|y-y_*(x)|^{s-2-\frac1p}\,
\|g_x\|_{W^{s-1,p}(-1,1)}.
\]
Substituting this into the left-hand side, we get
\[
\begin{aligned}
&\int_B\int_{-1}^1
|v(x,y)|^2\frac{|\p_y\phi(x,y)|}{|y-y_*(x)|}\,dy\,dx \\
&\qquad\lesssim
\int_B
\|g_x\|_{W^{s-1,p}(-1,1)}
\left(
\int_{-1}^1
|v(x,y)|^2 |y-y_*(x)|^{s-2-\frac1p}\,dy
\right)\,dx.
\end{aligned}
\]
Since $s>1+\frac1p$, the weight $|y-y_*(x)|^{s-2-\frac1p}$ is integrable on $(-1,1)$.
Hence, for each fixed $x\in B$, choosing $q>1$ large,
\[
\int_{-1}^1
|v(x,y)|^2 |y-y_*(x)|^{s-2-\frac1p}\,dy
\lesssim
\|v(x,\cdot)\|_{L^{2q}(-1,1)}^2.
\]
Therefore, it follows from Lemma \ref{lem:mixed-norm-H1} that
\begin{align*}
    \int_B\int_{-1}^1
|v(x,y)|^2\frac{|\p_y\phi(x,y)|}{|y-y_*(x)|}\,dy\,dx
&\lesssim
\int_B
\|v(x,\cdot)\|_{L^{2q}(-1,1)}^2
\|\p_y\phi(x,\cdot)\|_{W^{s-1,p}(-1,1)}
\,dx\\
&\lesssim \|\p_y\phi\|_{L^p_xW^{s-1,p}_y(\Om)} \|v\|_{L_x^{2p'}L_y^{2q}}^2 \\
&
 \lesssim \|\p_y \phi\|_{W^{s-1,p}(\Om)} \|v\|_{L_x^{2p'}L_y^{2q}}^2 \\
&\lesssim \|\phi\|_{W^{s,p}(\Om)} \|v\|_{L_x^{2p'}L_y^{2q}}^2\lesssim \|\phi\|_{W^{s,p}(\Om)} \|v\|_{H^1(\Om)}^2.
\end{align*}
\end{proof}

We now combine the preceding ingredients in the proof of Sobolev rigidity.
\begin{theorem}\label{thm:sobolev-rigidity}
For any $p>1$ and $s>1+\frac1p$, there exists $\ep_s>0$ such that if $\om \in W^{s,p}(\Omega)\cap L^2(\Om)$ satisfies, for some $c\in \R$,
\begin{equation}\label{eq:sob-rigid-steady}
(y+c+u^x)\p_x \om + u^y \p_y \om =0
\end{equation}
and
\begin{equation}\label{eq:sob-rigid-small}
\|\om\|_{W^{s,p}(\Omega)} \le \ep_s,
\end{equation}
then necessarily $\om\equiv 0$.
\end{theorem}
\begin{proof}
Let $\psi$ be the stream function, so that
\begin{equation}\label{eq:stream-vort-rigidity}
u=\nabla^\perp\psi,\qquad -\Delta \psi=\om,\qquad \psi|_{\p\Omega}=0.
\end{equation}

Since $\om\in W^{s,p}(\Omega)$ with $s>1+\frac1p$, standard elliptic estimates and Sobolev embedding yield
\[
\|\nabla u\|_{L^\infty(\Omega)}
\lesssim \|\om\|_{W^{s,p}(\Omega)}
\le \ep_s.
\]
Therefore, after choosing $\ep_s>0$ sufficiently small in
\eqref{eq:sob-rigid-small},  we may assume
\begin{equation}\label{eq:ux-small}
\|\p_y u^x\|_{L^\infty(\Omega)} \le \frac12.
\end{equation}

For each fixed $x\in \R$, define the function 
\[
F_x(y):=y+c+u^x(x,y), \qquad y\in[-1,1].
\]
By \eqref{eq:ux-small}, $y\mapsto F_x(y)$ is Lipschitz continuous and
\[
\p_y F_x(y)=1+\p_y u^x(x,y)\ge \frac12,
\]
so $F_x$ is strictly increasing in $y$.

For each $x\in \R$, we define a reference point $y_*(x)\in[-1,1]$ as follows:
\begin{itemize}
    \item if $F_x$ has a zero in $[-1,1]$, then $y_*(x)$ is that unique zero;
    \item if $F_x>0$ on $[-1,1]$, set $y_*(x)=-1$;
    \item if $F_x<0$ on $[-1,1]$, set $y_*(x)=1$.
\end{itemize}
In all cases, the monotonicity of $F_x$ gives
\begin{equation}\label{eq:critical-layer-lower}
|F_x(y)| \ge \frac12 |y-y_*(x)|
\qquad \text{for all }(x,y)\in \Omega.
\end{equation}

Next, since $u^y=-\p_x\psi$, we see that $u^y$ is a weak solution to the equation
\[
\Delta u^y=\p_x \om.
\]
Also, $\om\in L^2(\Omega)$ implies $u^y\in H^1(\Omega)$ by elliptic regularity for
\eqref{eq:stream-vort-rigidity}. Since $u^y|_{y=\pm1}=0$, integration by parts gives
\begin{equation}\label{eq:uy-energy}
\|\nabla u^y\|_{L^2(\Omega)}^2
= -\int_\Omega u^y\,\p_x \om\,dxdy.
\end{equation}
Since $ \om \in W^{s,p}$ with $s>1$, the steady equation \eqref{eq:sob-rigid-steady} holds in the $L^p(\Om)$ sense, and we may rewrite
\[
\p_x\om = - \frac{u^y \p_y \om}{y+c+u^x}
= -\frac{u^y \p_y \om}{F_x(y)}  \qquad\text{a.e. in }\Omega,
\]
and therefore \eqref{eq:uy-energy} gives
\begin{equation}\label{eq:energy-identity-sob}
\|\nabla u^y\|_{L^2(\Omega)}^2
\le \int_\Omega \frac{|u^y|^2 |\p_y \om|}{|F_x(y)|}\,dxdy.
\end{equation}

We next split the $x$-variable according to what happens at the reference point $y_*(x)$.

\medskip
\noindent
\textbf{Step 1. Splitting into \(A\) and \(B\).}
Using the reference function $F_x(y)$, we define two measurable sets $A,B \subset \R$,
\[
A:=\Bigl\{x\in \R:\ u^y(x,y_*(x))=0\Bigr\},
\qquad
B:=\R\setminus A.
\]

We claim that   for a.e. $x\in B$ 
\begin{equation}\label{eq:dyw-zero-on-B}
\p_y \om(x,y_*(x))=0.
\end{equation}

Indeed, by Lemma \ref{lemma:Sobolev_slicing} we may first restrict to those $x\in B$ where $ \p_x \om (x,\cdot)$ and  $\p_y \om(x,\cdot) \in W^{s-1,p}_y([-1,1]) $. Since the transport equation \eqref{eq:sob-rigid-steady} holds at the level $L^p(\Om)$, we may also assume $ x\in B$ is such that \eqref{eq:sob-rigid-steady} holds for  a.e.  $y\in [-1,1]$. Then  the one-dimensional Sobolev embedding (with \(s-1>1/p\)) gives, for a.e. $x\in B$,
\begin{equation}\label{eq:sob-rigid-steady ab_con}
 y\mapsto     \p_y \om(x,y )\quad \text{and} \quad y \mapsto \p_x \om(x,y) \quad \text{are     continuous  for $y\in [-1,1]$}
\end{equation}

For such $x\in B$, if $y_*(x)\in\{-1,1\}$, then $u^y(x,y_*(x))=0$ by the boundary condition implying $x \in A$, a contradiction. For those $x\in B$, we have $y_*(x)\in(-1,1)$, and hence by definition $F_x(y_*(x))=0$. Since \eqref{eq:sob-rigid-steady} holds in $L^p(\Om)$, by \eqref{eq:sob-rigid-steady ab_con}, for a.e. $x\in  B$, there holds
\begin{equation}\label{eq:sob-rigid-steady B condition}
u^y(x,y_*(x))\, \p_y \om(x,y_*(x))=0.
\end{equation}

We correspondingly decompose the right-hand side of \eqref{eq:energy-identity-sob} as
\[
I_A+I_B
:=
\int_{A\times(-1,1)} \frac{|u^y|^2 |\p_y \om|}{|F_x(y)|}\,dxdy
+
\int_{B\times(-1,1)} \frac{|u^y|^2 |\p_y \om|}{|F_x(y)|}\,dxdy.
\]

\medskip
\noindent
\textbf{Step 2. Estimate of \(I_A\).}
Fix $x\in A$. By construction, $u^y(x,y_*(x))=0$. Lemma \ref{lem:1d-zero-maximal} applied to $h(y)=u^y(x,y)$ gives
\[
|u^y(x,y)|
\lesssim |y-y_*(x)|\, g(x,y),
\qquad
g(x,y):=\mathcal M_y(\p_y u^y(x,\cdot))(y).
\]
Combining this with \eqref{eq:critical-layer-lower}, this yields
\[
\frac{|u^y|^2 |\p_y \om|}{|F_x(y)|}
\lesssim |u^y|\, g\, |\p_y \om|.
\]
Hence
\[
I_A \lesssim \int_\Omega |u^y|\, g\, |\p_y \om|\,dxdy.
\]

Since $s-1>\frac1p$, there exists $q>1$ such that
\begin{equation}\label{eq:sob-embed-q}
2<2q' < \infty
\qquad\text{and}\qquad
W^{s-1,p}(\Omega)\hookrightarrow L^{2q'}(\Omega).
\end{equation}
By H\"older's inequality with exponents $(2q,2,2q')$,
\[
I_A
\lesssim
\|u^y\|_{L^{2q}(\Omega)}
\|g\|_{L^2(\Omega)}
\|\p_y \om\|_{L^{2q'}(\Omega)}.
\]
By Lemma~\ref{lem:mixed-norm-H1} (applied with exponent $2q>2$) and Poincaré's inequality
in the $y$-variable,
\[
\|u^y\|_{L^{2q}(\Omega)}
\lesssim \|u^y\|_{H^1(\Omega)}
\lesssim \|\nabla u^y\|_{L^2(\Omega)}.
\]
Also, by the \(L^2\)-boundedness of the 1D Hardy--Littlewood maximal operator,
\[
\|g\|_{L^2(\Omega)}
\lesssim \|\p_y u^y\|_{L^2(\Omega)}
\le \|\nabla u^y\|_{L^2(\Omega)}.
\]
Finally, by \eqref{eq:sob-embed-q},
\[
\|\p_y \om\|_{L^{2q'}(\Omega)}
\lesssim \|\om\|_{W^{s,p}(\Omega)}
\le \ep_s.
\]
Therefore,
\begin{equation}\label{eq:IA-final}
I_A \lesssim \ep_s \|\nabla u^y\|_{L^2(\Omega)}^2.
\end{equation}

\medskip
\noindent
\textbf{Step 3. Estimate of \(I_B\).}
For a.e. $x\in B$, we have \eqref{eq:dyw-zero-on-B}. Hence by Lemma \ref{lem:weighted-critical-layer-sobolev} with $\phi=\om$, $v=u^y$, and $y_0(x)=y_*(x)$, we get
\[
I_B
\lesssim
\int_{B\times(-1,1)}
\frac{|u^y|^2 |\p_y \om|}{|y-y_*(x)|}\,dxdy
\lesssim
\|\om\|_{W^{s,p}(\Omega)} \|u^y\|_{H^1(\Omega)}^2.
\]
Using again Poincaré's inequality in \(y\) and the given assumption \eqref{eq:sob-rigid-small}, we infer
\begin{equation}\label{eq:IB-final}
I_B
\lesssim
\|\om\|_{W^{s,p}(\Omega)} \|u^y\|_{H^1(\Omega)}^2
\lesssim
\ep_s \|\nabla u^y\|_{L^2(\Omega)}^2.
\end{equation}

\medskip
\noindent
\textbf{Step 4. Conclusion.}
Combining \eqref{eq:energy-identity-sob}, \eqref{eq:IA-final}, and \eqref{eq:IB-final}, we obtain 
\[
\|\nabla u^y\|_{L^2(\Omega)}^2
\lesssim
\ep_s \|\nabla u^y\|_{L^2(\Omega)}^2.
\]
Choosing $\ep_s>0$ sufficiently small, it follows that
\[
\nabla u^y\equiv 0.
\]
Since $u^y=0$ on $y=\pm1$, we conclude that
\[
u^y\equiv 0.
\]
By incompressibility, $
\p_x u^x = -\p_y u^y =0$,
and therefore $u^x$ is independent of $x$. Since $u^x\in L^2(\Omega)$ and $\Omega$ is unbounded
in the $x$-direction, this implies $ u^x\equiv 0 $. 
Hence $u\equiv 0$, and therefore
\[
\om=\p_x u^y-\p_y u^x\equiv 0.
\]
This completes the proof.
\end{proof}

\subsection{Non-existence in the H\"older scale}
We next prove the analogous rigidity statement in the H\"older scale. The argument follows the same critical-layer mechanism as in the Sobolev case, with the fractional Sobolev control replaced by the pointwise decay supplied by \(C^{1,\alpha}\) regularity.

\begin{theorem}\label{thm:holder-rigidity}
For any $0<\alpha\le 1$, there exists $\ep_\alpha>0$ such that if
$\om\in C^{1,\alpha}(\Omega)\cap L^2(\Omega)$ satisfies, for some $c\in\R$,
\begin{equation}\label{eq:holder-rigid-steady}
(y+c+u^x)\p_x\om+u^y\p_y\om=0
\end{equation}
and
\begin{equation}\label{eq:holder-rigid-small}
\|\om\|_{C^{1,\alpha}(\Omega)}\le \ep_\alpha,
\end{equation}
then necessarily $\om\equiv 0$.
\end{theorem}

\begin{proof}
As in the proof of Theorem \ref{thm:sobolev-rigidity}, choose $\ep_\alpha>0$ sufficiently small so that
\begin{equation}\label{eq:holder-ux-small}
\|\p_y u^x\|_{L^\infty(\Omega)}\le \frac12.
\end{equation}
For each fixed $x\in\R$, define
\[
F_x(y):=y+c+u^x(x,y), \qquad y\in[-1,1].
\]
Then
\[
\p_y F_x(y)=1+\p_y u^x(x,y)\ge \frac12,
\]
so $F_x$ is strictly increasing in $y$. Define $y_*(x)\in[-1,1]$ exactly as in the proof of Theorem \ref{thm:sobolev-rigidity}. In particular,
\begin{equation}\label{eq:holder-critical-layer-lower}
|F_x(y)|\ge \frac12 |y-y_*(x)|
\qquad\text{for all }(x,y)\in\Omega.
\end{equation}

Since $u^y|_{y=\pm1}=0$ and $\Delta u^y=\p_x\om$, integrating by parts gives
\[
\|\nabla u^y\|_{L^2(\Omega)}^2
=
-\int_\Omega u^y\,\p_x\om\,dxdy
\le
\int_\Omega \frac{|u^y|^2\,|\p_y\om|}{|F_x(y)|}\,dxdy.
\]
We decompose the right-hand side as in the proof of Theorem \ref{thm:sobolev-rigidity}:
\[
I_A+I_B
:=
\int_{A\times(-1,1)} \frac{|u^y|^2\,|\p_y\om|}{|F_x(y)|}\,dxdy
+
\int_{B\times(-1,1)} \frac{|u^y|^2\,|\p_y\om|}{|F_x(y)|}\,dxdy,
\]
where
\[
A:=\Bigl\{x\in \R:\ u^y(x,y_*(x))=0\Bigr\},
\qquad
B:=\R\setminus A.
\]
For $x\in B$, exactly as before,
\begin{equation}\label{eq:holder-dyw-zero}
\p_y\om(x,y_*(x))=0.
\end{equation}

\medskip
\noindent
\textbf{Estimate of $I_A$.}
For $x\in A$, Lemma \ref{lem:1d-zero-maximal} applied to $h(y)=u^y(x,y)$ yields
\[
|u^y(x,y)|
\lesssim
|y-y_*(x)|\,g(x,y),
\qquad
g(x,y):=\mathcal M_y(\p_y u^y(x,\cdot))(y).
\]
Together with \eqref{eq:holder-critical-layer-lower},
\[
\frac{|u^y|^2\,|\p_y\om|}{|F_x(y)|}
\lesssim
|u^y|\,g\,|\p_y\om|.
\]
Hence
\[
I_A
\lesssim
\|\p_y\om\|_{L^\infty(\Omega)}\int_\Omega |u^y|\,g\,dxdy.
\]
Since the maximal operator is bounded on $L^2$ in the $y$-variable,
\[
\|g\|_{L^2(\Omega)}
\lesssim
\|\p_y u^y\|_{L^2(\Omega)}
\le
\|\nabla u^y\|_{L^2(\Omega)}.
\]
Moreover, since $u^y=0$ on $y=\pm1$, the Poincar\'e inequality in the $y$-variable gives
\[
\|u^y\|_{L^2(\Omega)}
\lesssim
\|\p_y u^y\|_{L^2(\Omega)}
\le
\|\nabla u^y\|_{L^2(\Omega)}.
\]
Therefore,
\[
I_A
\lesssim
\|\om\|_{C^{1,\alpha}(\Omega)}\|u^y\|_{L^2(\Omega)}\|g\|_{L^2(\Omega)}
\lesssim
\ep_\alpha \|\nabla u^y\|_{L^2(\Omega)}^2.
\]

\medskip
\noindent
\textbf{Estimate of $I_B$.}
For $x\in B$, by \eqref{eq:holder-dyw-zero} and the $C^{1,\alpha}$ regularity of $\om$,
\[
|\p_y\om(x,y)|
\le
\|\om\|_{C^{1,\alpha}(\Omega)}\,|y-y_*(x)|^\alpha.
\]
Using \eqref{eq:holder-critical-layer-lower},
\[
\frac{|\p_y\om(x,y)|}{|F_x(y)|}
\lesssim
\|\om\|_{C^{1,\alpha}(\Omega)}\,|y-y_*(x)|^{-1+\alpha}.
\]
Hence
\[
I_B
\lesssim
\|\om\|_{C^{1,\alpha}(\Omega)}
\int_B\int_{-1}^1 |u^y(x,y)|^2 |y-y_*(x)|^{-1+\alpha}\,dy\,dx.
\]
Since $\alpha>0$, the weight $|y-y_*(x)|^{-1+\alpha}$ is integrable on $(-1,1)$ uniformly in $x$, so
\[
\int_{-1}^1 |u^y(x,y)|^2 |y-y_*(x)|^{-1+\alpha}\,dy
\lesssim
\|u^y(x,\cdot)\|_{L^\infty(-1,1)}^2.
\]
Using the 1D Sobolev embedding and the boundary conditions $u^y(x,\pm1)=0$,
\[
\|u^y(x,\cdot)\|_{L^\infty(-1,1)}
\lesssim
\|u^y(x,\cdot)\|_{H^1(-1,1)}
\lesssim
\|\p_y u^y(x,\cdot)\|_{L^2(-1,1)}.
\]
Therefore,
\[
I_B
\lesssim
\|\om\|_{C^{1,\alpha}(\Omega)}
\int_\R \|\p_y u^y(x,\cdot)\|_{L^2(-1,1)}^2\,dx
\lesssim
\ep_\alpha \|\nabla u^y\|_{L^2(\Omega)}^2.
\]

Combining the estimates for $I_A$ and $I_B$, we obtain
\[
\|\nabla u^y\|_{L^2(\Omega)}^2
\lesssim
\ep_\alpha \|\nabla u^y\|_{L^2(\Omega)}^2.
\]
Choosing $\ep_\alpha>0$ sufficiently small, it follows that
\[
\nabla u^y\equiv 0.
\]
This  completes the proof. 
\end{proof}


\bibliographystyle{plain}
\bibliography{1}

\end{document}